\documentclass{amsart}
\usepackage{amsmath,amsfonts,amsthm}
\usepackage{amssymb,latexsym}
\usepackage[colorlinks]{hyperref}
\usepackage{graphicx}
\usepackage[all]{xy}
\usepackage{layout}

\theoremstyle{plain}
\newtheorem{theorem}{Theorem}[section]
\newtheorem{lemma}[theorem]{Lemma}
\newtheorem{proposition}[theorem]{Proposition}
\newtheorem{corollary}[theorem]{Corollary}

\theoremstyle{definition}

\newtheorem{definition}[theorem]{Definition}

\theoremstyle{remark}

\newtheorem*{remark}{Remark}
\numberwithin{equation}{section}

\newcommand{\M}{\mathbb M_{k(n)}(E)}
\newcommand{\MMM}{\mathbb M_{k(n)}(\mathfrak A^{**})}
\newcommand{\MM}{\mathbb M_{k(n)}(\mathfrak A)}

\begin{document}

\title{Finite approximation properties of
$C^{*}$-modules}

\author{Massoud Amini}

\address{Faculty of Mathematical Sciences, Tarbiat Modares University\\ Tehran 14115-134, Iran\\ School of Mathematics,
	Institute for Research in Fundamental Sciences (IPM)\\Tehran 19395-5746, Iran\\ mamini@modares.ac.ir,  mamini@ipm.ir}

\address{current address: STEM Complex, 150 Louis-Pasteur Pvt,
	Ottawa, ON, Canada K1N 6N5}

\thanks{The author was partly supported by a grant from IPM (No. 94430215).}


\begin{abstract}
We study the notions of nuclearity and exactness  for
module maps on $C^{*}$-algebras which are $C^*$-module over another $C^*$-algebra
with compatible actions and examine finite approximation properties of such $C^*$-modules. We prove module versions of the results of Kirchberg and Choi-Effros. As a concrete example we extend the finite dimensional approximation properties of reduced $C^*$-algebras and  von Neumann algebras on countable discrete groups to  these operator algebras on countable inverse semigroups with the module structure coming from the action of the $C^*$-algebras on the subsemigroup of idempotents.
\end{abstract}

\maketitle

\section{introduction}\label{s1}

Finite dimensional approximation properties of $C^*$-algebras is core subject in modern theory of operator algebras \cite{bo}. These  include important notions such as nuclearity, exactness and weak expectation property (WEP). The results in this direction are obtained based on the classical extension and dilation results due to Arveson, Wittstock and Stinespring. Some of these results are also valid for $C^*$-module maps \cite{w2}. It is desirable then to consider the finite dimensional approximation properties of $C^*$-modules. The motivation is two fold: A finite dimensional approximation scheme for $C^*$-morphisms is as follows:

\begin{center}
$\xymatrix @!0 @C=4pc @R=3pc { A \ar[rr]^{\theta} \ar[rd]^{\varphi_n} && B  \\ & \mathbb M_{k_n}(\mathbb C) \ar[ur]^{\psi_n}}$
\end{center}

\noindent where $A$ and $B$ are $C^*$-algebras and $\varphi_n$ and $\psi_n$ are contractive completely positive (c.c.p.) maps. The case $A=B$ and $\theta=\mathrm{id}_A$ is of special interest.

There are situations that such an approximate decomposition is needed through $\mathbb M_{k_n}(N)$ for a $C^*$-algebra or von Neumann algebra $N$. One instance is the notion of a correspondence $H={}_MH_N$ between von Neumann algebras $M$ and $N$. It is shown that $H$ is left amenable iff there is a net of c.p. maps $\theta_n: M\to M$ converging point-ultraweakly to the identity on $M$ such that each $\theta_n$ is a finite sum of compositions of $\psi_n: M\to \mathbb M_{k_n}(N)$ and $\varphi_n: \mathbb M_{k_n}(N)\to M$ of certain specific form \cite[Theorem 2.2]{del}. In this situation $N$ is usually a subalgebra of $M$, but the approximate decomposition is meaningful if $M$ is an $N$-module and the maps are module maps. It is worthwhile to study the possibility of such approximate decompositions in general. Since $\mathbb M_{k_n}(N)$ is a finitely generated module over $N$, we may call the sequence of such  decompositions a ``finite approximation'' in the category of $N$-modules, where ``finite'' refers to ``finitely generated''. The reader should be warned that here we do not deal with a ``finite dimensional approximation'', even in the module sense, as for instance $\mathbb M_{k_n}(N)$ does not necessarily have finite uniform (Goldie) dimension as an $N$-module in general (see for instance, \cite[Chapter 6]{lam}). 

On the other hand, there are concrete examples where it is not know when a decomposition through $M_{k_n}(\mathbb C)$ exists. For instance, as far as I know, it is not known when the reduced $C^*$-algebra $C^*_r(S)$ of an inverse semigroup $S$ is nuclear (i.e., when the identity map on $C^*_r(S)$ approximately decomposes in point-norm through matrix algebras.) This is of course equivalent to the amenability of the corresponding universal groupoid, but equivalent conditions on $S$ is only known in special cases (for instance, if $S$ is $E$-unitary, it is known to be equivalent to the left amenability of $S$.) It is desirable to take some natural module structures on a suitable $C^*$-algebra $\mathfrak A$ into account and find approximate decompositions through $M_{k_n}(\mathfrak A)$. Here we show that under quite natural actions of the subsemigroup $E_S$ of idempotents of $S$, $C^*_r(S)$ is a $C^*_r(E_S)$-module and a point-norm approximate decomposition of the identity map on $C^*_r(S)$ through algebras $M_{k_n}(C^*_r(E_S))$ is possible iff $S$ is left amenable.

This paper  studies   finite dimensional approximation properties of  $C^*$-modules in more details. The paper is organized as follows: In section 2 we introduce the notion of nuclear module morphisms for
a $C^{*}$-algebra $A$ which is $C^*$-module over another $C^*$-algebra $\mathfrak A$
with compatible actions, and use it to define nuclearity and exactness in the category of $C^*$-modules. In section 3, we study tensor products in the category of $C^*$-modules and extend Takesaki theorem (Theorem \ref{tak}) and "the trick" \cite[section 3.6]{bo}.  We also extends classical results of Kirchberg (Theorem \ref{kirchberg}) and Choi-Effros, Kirchberg (Theorem \ref{kirchberg2}) in the module setting.

For the rest of this paper, we fix a $C^*$-algebra $\mathfrak A$ and let $A$ be a $C^*$-algebra and a right Banach
$\mathfrak A$-module (that is, a module with contractive right action) with compatibility conditions,
\begin{equation*}
 (ab)\cdot\alpha=a(b\cdot\alpha),\,\,   a\cdot\alpha\beta=(a\cdot\alpha)\cdot\beta,
\end{equation*}
for each $a,b\in A $ and $\alpha, \beta\in \mathfrak A.$ In this case, we  say that $A$ is a (right) $\mathfrak A$-$C^*$-module, or simply a $C^*$-module (it is then understood that the algebra and module structures on $A$ are compatible in the above sense). A $C^*$-subalgebra which is also an $\mathfrak A$-submodule is simply called a $C^*$-submodule.

We say that the action is {\it non-degenerate} if, given $\alpha\in\mathfrak A$, $ a\cdot\alpha=0$ for each $a\in A$ implies that $\alpha=0$. We need this condition when we work with the unitization (see the next section).

We write $\mathfrak A\cdot A$ for the closed linear span of the set of elements of the form $ a\cdot\alpha$, with $\alpha\in\mathfrak A, \ a\in A$. We say that $A$ is {\it neo-unital} (as an $\mathfrak A$-module) if $ A\cdot\mathfrak A=A$. In this case, if $\mathfrak A$ is a unital $C^*$-algebra, then $ a\cdot 1_{\mathfrak A}=a$, for $a\in A$. When $\mathfrak A$ is a unital $C^*$-algebra, we say that $A$ is unital (as an $\mathfrak A$-module), if $A$ is a unital $C^*$-algebra and a neo-unital $\mathfrak A$-module.

Let $Z(A)$ be the center of $A$. An element $a\in A$ is called $\mathfrak A$-central if $ a\cdot\alpha\in Z(A)$, for each $\alpha\in \mathfrak A$. We say that $\mathfrak A$ acts centrally on $A$, or $A$ is a {\it central} $\mathfrak A$-module, if $Z(A)$ is a submodule of $A$. When $\mathfrak A$ is a unital $C^*$-algebra, we say that $A$ is unital (as an  $\mathfrak A$-module), if $A$ is a unital $C^*$-algebra and a neo-unital $\mathfrak A$-module. In this case, $A$ is a central $\mathfrak A$-module if and only if $1_A$ is a central element of $A$. this is also equivalent to the following compatibility condition:
$$(a\cdot\alpha)(b\cdot\beta)= (ab)\cdot\alpha\beta,$$
for each $\alpha, \beta\in \mathfrak A$, each $a\in Z(A)$, and $b\in A$.

In some cases we have to work with operator $\mathfrak A$-modules with no algebra structure (and in particular with certain Hilbert $\mathfrak A$-modules). If $E, F$ are operator $\mathfrak A$-modules, a module map $\phi: E\to F$ is a continuous linear map which preserves the left $\mathfrak A$-module action.

A Hilbert space $H$ with a non-degenerate representation $\rho$ of $\mathfrak A$ in $H$ is called an $\mathfrak A$-Hilbert space. We may regard $H$ as a left $\mathfrak A$-module (where $\mathfrak A$ acts through $\alpha\cdot\xi :=\rho(\alpha)\xi$). In this paper we work with two types of representations of operator $\mathfrak A$-modules. One is representation  in left $\mathfrak A$-Hilbert $C^*$-modules and the other is representation  in $\mathfrak A$-Hilbert spaces. We usually work with the former, but from time to time, we also need to work with the latter.

Throughout this paper, we use the notations $\mathbb K(X)$ and $\mathbb B(X)$ to denote the set of compact and bounded adjointable linear operators on a (left) Hilbert  $\mathfrak A$-module $X$ (see \cite{lan} and \cite{mt}, for details). Note that $\mathbb B(X)$ is a  right $\mathfrak A$-$C^*$-module under the action
$(T\cdot\alpha)(\xi)=T(\alpha\cdot\xi),$ for $\alpha\in \mathfrak A, \xi\in X, t\in \mathbb B(X)$.

\section{nuclear module maps}

Let $A, B$ be operator $\mathfrak A$-modules and $E$ be an operator system and operator $\mathfrak A$-module. In this section we define a notion of $E$-nuclearity for module maps $\theta: A\to B$. Two cases of particular interest are when $E=\mathfrak A$ or $\mathfrak A^{**}$. When $A\subseteq B$ is a submodule, a conditional expectation $\mathbb E: B\to A$, which is also a $\mathfrak A$-module map, is called a $\mathfrak A$-conditional expectation.

We freely use the abbreviations and notations of \cite{bo}, in particular, c.p., u.c.p., and c.c.p. stand for completely positive, unital completely positive, and contractive completely positive, respectively.

The notion of $E$-nuclear maps are defined based on the possibility of approximate decomposition through $\mathbb M_{k_n}(E)$. It turns out that it would be too restrictive to assume that the maps of approximate decomposition are module maps. Instead, we specify a class of "admissible" maps which could make a decomposition.

\begin{definition}
The class of {\it admissible} c.p. maps between operator $\mathfrak A$-modules is characterized through the following set of rules:

$(i)$ a c.p. module map is admissible,

$(ii)$ if a c.p. map $\phi: A\to B$ is admissible, then so are the maps $x\mapsto \phi(uxu^*)$ and $x\mapsto v\phi(x)v^*$, for each $u\in A$ and $v\in B$,

$(iii)$ the c.p. maps $\phi: A\to B$ of the form $\phi(x)=\rho(x)u$, where $\rho\in A^*_{+}$ and $u\in B^{+}$, are admissible,

$(iv)$ the composition, positive multiples, or finite sums of admissible c.p. maps are admissible.

\end{definition}

\begin{definition}
A module map $\theta: A\to B$ is called $E$-{\it nuclear} if there are c.c.p. admissible maps $\varphi_n: A\to \M$ and $\psi_n: \M\to B$ such that $\psi_n\circ\varphi_n\to \theta$ in point-norm topology, that is,
$$\|\psi_n\circ\varphi_n(a)-\theta(a)\|\to 0,$$
for each $a\in A$.
When $B$ is a von Neumann algebra, a module map $\theta: A\to B$ is called $E$-{\it weakly nuclear} if there are c.c.p. admissible maps $\varphi_n: A\to \M$ and $\psi_n: \M\to B$ such that $\psi_n\circ\varphi_n\to \theta$ in point-ultraweak topology.
\end{definition}

A nuclear map is automatically c.c.p. The following lemma is proved similar to the classical case \cite[Exercises 2.1.3-4,7-8]{bo}.

\begin{lemma} \label{nuc}
Let $\theta: A\to B$ be a c.c.p. module map.

$(i)$ $($restriction$)$ If $C\subseteq A$ is a $C^*$-submodule and $\theta$ is $E$-nuclear, then so is its restriction $\theta|_C: C\to B$.

$(ii)$ $($dependence on range$)$ If  $\theta$ is $E$-nuclear and $D\subseteq B$ is a $C^*$-submodule with $\theta(A)\subseteq D$, then under any of the following conditions, $\theta: A\to D$  is $E$-nuclear:

\hspace{.3cm} $(ii$-$1)$ There is a conditional expectation $\mathbb E: B\to D$,

\hspace{.3cm} $(ii$-$2)$ There is a sequence of c.c.p. admissible maps $\mathbb E_n: B\to D$ such that $\mathbb E_n\to id_D$ on $D$ in the point-norm topology.

$(ii)$ $($composition$)$ If $C$ is a $C^*$-module and $\sigma: B\to C$ is a c.c.p. module map such that $\theta$ or $\sigma$ is $E$-nuclear, then so is the composition $\sigma\circ\theta$.
\end{lemma}

In particular, if $id_A: A\to A$ is $E$-nuclear, then so is any c.c.p. module map $\theta: A\to B$, for any $C^*$-module $B$.

Let $X$ be a left Hilbert $\mathfrak A$-module. An adjointable operator $q: X\to X$ is called a {\it projection} if $q^2=q=q^*$. A projection is said to be {\it finite rank} if its range is a finitely-generated $\mathfrak A$-submodule. In this case rank$(q)$ is the dimension of its range as an $\mathfrak A$-module (that is the minimum number of generators). A typical example is provided by the element $\theta_{x,x}$, defined by 
$$\theta_{x,x}(y):= \langle x,y\rangle\cdot x \ \ (x,y\in X),$$
whenever $\langle x,x\rangle\in\mathfrak A$ is a projection. The linear span of finite rank projections is known to be dense in $\mathbb K(X)$ \cite[page 10]{lan}. When $\mathfrak A$ is unital and $X=H\otimes \mathfrak A$, then for a finite rank projection $p\in\mathbb K(H)$ in the usual sense, $q:=p\otimes 1_{\mathfrak A}$ is a finite rank projection in $\mathbb K(H\otimes \mathfrak A)=\mathbb K(H)\otimes \mathfrak A$ and $q(H\otimes \mathfrak A)=pH\otimes \mathfrak A\simeq \ell^2_n\otimes \mathfrak A$, where $n=$rank$(p)$ \cite[page 10]{lan}. In this case, $\mathfrak A$ acts only on the second leg, namely,
$$\alpha\cdot(\xi\otimes\beta)=\xi\otimes \alpha\beta, \ \ (\xi\in H, \alpha,\beta\in \mathfrak A).$$

\begin{lemma} \label{fr}
	If $\mathfrak A$ is unital, for any Hilbert space $H$, there is a net $(q_i)$ of finite rank projections in $\mathbb B(H\otimes \mathfrak A)$, increasing to 1 in the strong operator topology, such that $q_i\mathbb B(H\otimes \mathfrak A)q_i\simeq \mathbb M_{n_i}(\mathfrak A)$, as $C^*$-algebras and left $\mathfrak A$-modules, where $n_i={\rm rank}(q_i)$.  
\end{lemma}
\begin{proof}
	Choose an increasing net $(p_i)$ of finite rank projections in $\mathbb B(H)$, increasing to 1 in SOT. Put $q_i:=p_i\otimes 1_{\mathfrak A}$ and observe that  
	 $$q_i\mathbb K(H\otimes \mathfrak A)q_i=\mathbb{K}(p_iH\otimes \mathfrak A)\simeq \mathbb M_{n_i}(\mathfrak A),$$ where the last isomorphism is shown in \cite[Proposition 2.2.2$(ii)$]{mt}, with  $n_i={\rm rank}(q_i) ={\rm rank}(p_i)$. Next, note that 
	 $$q_i\mathbb B(H\otimes \mathfrak A)q_i=\mathbb B(p_iH\otimes \mathfrak A)=\mathbb K(p_iH\otimes \mathfrak A)=q_i\mathbb K(H\otimes \mathfrak A)q_i,$$
	 where the second equality follows from \cite[Theorem 2.4]{lan} and the fact that the algebra $\mathbb K(p_iH\otimes \mathfrak A)=\mathbb  K(p_iH)\otimes \mathfrak A$ is unital (as $p_i$ is finite rank). Finally, it is easy to check that all the natural isomorphisms involved are also module maps, since for basic tensors $x:=\xi\otimes \beta\in H\otimes \mathfrak A$, 
	 $$q_i(\alpha \cdot x)=(p_i\otimes 1_{\mathfrak A})(\alpha\cdot(\xi\otimes\beta) )=(p_i\xi\otimes \alpha\beta)=\alpha\cdot(p_i\xi\otimes \beta) =\alpha \cdot q_i(x) ,$$
	 for $\xi\in H$ and $\alpha,\beta\in \mathfrak A$. 	  	  
\end{proof}

The above result does not hold for general Hilbert $\mathfrak A$-modules. Indeed, in general $\mathbb B(X)$ is only a $C^*$-algebra and may fail to contain any non-trivial projection. If  $X$ is a self dual left Hilbert $\mathfrak A$-module with a  standard Riesz basis, one could recover part of the proof by finding a net of finite rank projections increasing SOT to 1. Recall that for a finite or countable
index set $J$, a  sequence $\{x_j : j \in J\}$ of elements in a right Hilbert $\mathfrak A$-module $X$
is called a frame if there are  constants $C, D > 0$ such that
$$C\langle x,x\rangle\leq \langle x, x_j\rangle\langle x_j, x\rangle \leq D \langle x,x\rangle,$$
for every $x\in X$. Such a frame is  normalized if $C = D = 1$ and standard if moreover the sum in the middle of the above inequality converges in norm. A Riesz basis 
is a frame which is also generating a dense subset of $X$. When we moreover have the uniqueness of expansion (i.e., an $\mathfrak A$-linear
finite combination $\sum_i \alpha_j\cdot x_j$ is zero if and only if every summand $\alpha_j\cdot x_j$ equals zero) this is called a standard Riesz basis for $X$ \cite[Definition 2.1]{fl}. Now since $X$ is a self dual, $\mathbb B(X)$ is a von Neumann algebra by \cite[Proposition 3.10]{pa}, hence there is an increasing net $\{q_i\}\subseteq \mathbb B(X)$ of finite rank projections, converging to the identity in the strong operator topology. However, though $q_i\mathbb B(X)q_i=\mathbb B(q_i X)$ and $q_iX$ is a finitely generated $\mathfrak A$-module, in general it is not possible to identify $q_iX$ with $\ell^2_{n_i}\otimes \mathfrak A$, for $n_i={\rm rank}(q_i)$.  

\begin{proposition}
	Let $\mathfrak A$ be a von Neumann algebra and $A$ be a von Neumann algebra and a module with compatible actions. Let $H$ be a Hilbert space. Then any faithful representation $A\hookrightarrow \mathbb B(H\otimes \mathfrak A)$ is $\mathfrak A$-weakly nuclear.
\end{proposition}
\begin{proof}
	Use Lemma \ref{fr} to choose an increasing net $\{q_i\}\subseteq \mathbb B(H\otimes \mathfrak A)$ of finite rank projections, converging to the identity in the strong operator topology such that $q_i\mathbb B(H\otimes\mathfrak A)q_i$ could be identified with $\mathbb M_{n_i}(\mathfrak A)$, for $n_i=$rank$(q_i)$, and note that the corresponding c.c.p. compression map $\varphi_i: A\to \mathbb M_{n_i}(\mathfrak A)$ is a module map. This composed with the inclusion map (after identification) $\psi_i: \mathbb M_{n_i}(\mathfrak A)\to \mathbb B(H\otimes\mathfrak A)$ gives an approximate decomposition in the ultraweak topology of $\mathbb B(H\otimes \mathfrak A)$. Indeed, by the proof of \cite[Proposition 3.10]{pa}, $\mathbb B(H\otimes \mathfrak A)\subseteq ((H\otimes \mathfrak A)\otimes (\bar H\otimes\mathfrak A^{\rm op})\otimes \mathfrak A_{*})^{*}$, where $\bar H$ is the  Hilbert space of the contragredient representation and $\mathfrak A^{\rm op}$ is the opposite algebra with reversed product and conjugate scalar product. Given a basic tensor $x=\xi\otimes\alpha$, let us put $\bar x:=\bar\xi\otimes \alpha^*$, and extend the bar operation to the whole $H\otimes \mathfrak A$ (see, \cite[page 455]{pa} for the general case). Next, let us  identify $A$ with its image in $\mathbb B(H\otimes \mathfrak A)$, and observe that, for each $x,y\in H\otimes\mathfrak A$,  $a\in A$, and $\phi\in\mathfrak A_*$, 
	$$(x\otimes\bar y\otimes\phi)(\psi_i\circ\varphi_i(a)-a)=\phi(\langle(\psi_i\circ\varphi_i(a)-a)x,y\rangle)\to 0,$$
	as $i\to\infty$. Indeed, for basic tensors $x=\xi\otimes\alpha$ and $y=\eta\otimes\beta$, we have
	\begin{align*}
		\phi(\langle(\psi_i\circ\varphi_i(a)-a)x,y\rangle)&=\phi(\langle((q_iaq_i-a)x,y\rangle)\\& =\phi(\langle a\big((p_i\xi-\xi)\otimes\alpha\big),\eta\otimes\beta\rangle)\to 0,  
	\end{align*}
	as $p_i\to 1$ in SOT. The result now follows by a standard density argument.
\end{proof}

\begin{definition}\label{exa}
Let $A$ be a $C^*$-algebra and a right $\mathfrak A$-module with compatible actions, let $E$ be an operator system and operator $\mathfrak A$-module. Then $A$ is called $E$-{\it nuclear} if the identity map on $A$ is $E$-nuclear. Also $A$ is called $E$-{\it exact} if there is $E$-nuclear faithful representation $\pi: A\to \mathbb B(X)$, for some left Hilbert $\mathfrak A$-module $X$. When $\mathfrak A$ is a von Neumann algebra and $A$ is a von Neumann algebra and a module with compatible actions, then $A$ is called $E$-{\it semidiscrete} if the identity map on $A$ is $E$-weakly nuclear.
\end{definition}

By Lemma \ref{nuc}$(ii)$, if  $\pi: A\to \mathbb B(X)$ is a faithful representation of $A$ in a left $\mathfrak A$-Hilbert module $X$, then $A$ is $E$-nuclear ($E$-exact) if and only if $\pi$ is $E$-nuclear ($E$-exact) as a module map from $A$ to $\pi(A)$ (to $\mathbb B(X)$). 

\begin{definition}
A c.p. module map $\theta: A\to B$ is called $E$-{\it factorable} if there is a positive integer $n$ and c.p. admissible maps $\varphi: A\to \mathbb M_n(E)$ and $\psi: \mathbb M_n(E)\to B$ such that $\psi\circ\varphi=\theta$.
\end{definition}

\begin{lemma} \label{convex}
Let $\mathbb B_{\mathfrak A}(A)$ be the Banach space of all bounded linear module maps on $A$, then

$(i)$ each convex subset $\mathcal C\subseteq \mathbb B_{\mathfrak A}(A)$  has the same point-norm and point-weak closures in $\mathbb B_{\mathfrak A}(A)$,

$(ii)$ the set of $E$-factorable module maps from $A$ to $B$ is convex. The same holds for the set of module maps which are $E$-factorable through c.c.p. admissible maps.

\end{lemma}
\begin{proof}
Like in the classical case \cite[2.3.4, 2.3.6]{bo}, part $(i)$ follows from Hahn-Banach theorem and part $(ii)$ follows by standard algebraic manipulations.
\end{proof}

When $\mathfrak A$ is a unital $C^*$-algebra and the action is non-degenerate, then $\tilde A:=A\times \mathfrak A$ is called the (minimal) {\it unitization} of $A$. It is a unital $C^*$-algebra under the norm $\|(a,\alpha)\|=\sup\{\|ab+ b\cdot\alpha\|: b\in A, \|b\|=1\}$ and multiplication
$$(a,\alpha)(b,\beta)=(ab+ b\cdot\alpha+ a\cdot\beta, \alpha\beta),$$
with unit $(0, 1_{\mathfrak A})$. Then $A$ is identified with a closed ideal in $\tilde A$. The module action of $\mathfrak A$ on $\tilde A$ is defined by $ (a,\beta)\cdot\alpha=(a\cdot\alpha, \beta\alpha)$. Clearly $\tilde A$ is neo-unital as a $\mathfrak A$-module, if $A$ is so. One should note that $\tilde A$ is not the same as the direct sum $A\oplus \mathfrak A$ (whose elements are denoted by $a\oplus\alpha$). In the rest of this paper, whenever we work with the minimal unitization, we assume that the action is non-degenerate.

The next result extends \cite[2.2.1-2.2.4]{bo} with similar proofs (except that here we should also take care of the module actions).

\begin{proposition}\label{unital}
Let $\mathfrak A$ be a unital $C^*$-algebra, $B$ be a central $\mathfrak A$-module, and $\theta: A\to B$ be a c.c.p. module map.

$(i)$ If $B$ is unital, $\theta: A\to B$ extends to a u.c.p. module map
$$ \tilde\theta: \tilde A\to B;\ \ \tilde\theta(a, \alpha):=\theta(a)+  1_B\cdot\alpha.$$
In general, $\theta: A\to B$ extends to a u.c.p. module map
$$ \tilde\theta: \tilde A\to \tilde B;\ \ \tilde\theta(a, \alpha):=(\theta(a),  1_{\tilde B}\cdot\alpha).$$

$(ii)$ If $E$ is unital and $\theta$ is $E$-nuclear, then  so is $\tilde\theta$.
\end{proposition}
\begin{proof}
$(i)$ We have to show that $\tilde \theta$ is c.p. map. The action of $\mathfrak A$ on $A$ extends to an action on $A^{**}$ which satisfies the compatibility conditions, and $A^{**}$ is a neo-unital $\mathfrak A$-module. Thus $\tilde A$ could be identified with $A+ 1_{A^{**}}\cdot\mathfrak A\subseteq A^{**}$, and $\tilde\theta$ is the restriction of the c.p. map $\theta^{**}$, and so  $\tilde\theta: \tilde A\to B^{**}$ is a c.p. map. Next we have the matrix equation
$$\tilde\theta_n[a_{ij}+ 1_{A**}\cdot\alpha_{ij}]-(\theta^{**})_n[a_{ij}+ 1_{A**}\cdot\alpha_{ij}]={\rm diag}\big(1_B-\theta^{**}(1_{A^{**}})\big)[ 1_{B}\cdot\alpha_{ij}].$$
If $[a_{ij}+ 1_{A**}\cdot\alpha_{ij}]\in \mathbb M_n(\tilde A)_{+}$, then $[ 1_{A^{**}}\cdot\alpha_{ij}]\in \mathbb M_n(A^{**})_{+}$, since $\mathbb M_n(\tilde A\cdot 1_{A^{**}})\subseteq \mathbb M_n(A^{**})$ could be identified with the quotient of $\mathbb M_n(A^{**})$ by the closed ideal $\mathbb M_n(\mathbb C\cdot 1_{A^{**}})$. Thus
$$[ 1_{B}\cdot\alpha_{ij}]=(\tilde\theta)_n[ 1_{A**}\cdot\alpha_{ij}]\in \mathbb M_n(B^{**})_{+}.$$
On the other hand, by assumption $ 1_B\cdot\alpha_{ij}$ is in the center of $B$, and in particular,
\begin{align*} ( 1_{B}\cdot\alpha_{ij})\big(1_B-\theta^{**}(1_{A^{**}})\big)
  &=\big(1_B-\theta^{**}(1_{A^{**}})\big)( 1_{B}\cdot\alpha_{ij}),
  \end{align*}
 hence the positive matrices $[ 1_{B}\cdot\alpha_{ij}]$ and ${\rm diag}\big(1_B-\theta^{**}(1_{A^{**}})\big)$ commute and their product is positive in $\mathbb M_n(B^{**})$. Therefore,
$$\tilde\theta_n[a_{ij}+ 1_{A**}\cdot\alpha_{ij}]\geq(\theta^{**})_n[a_{ij}+ 1_{A**}\cdot\alpha_{ij}]\geq 0,$$
as required.

$(ii)$ We may assume that $B$ is unital. By assumption, let  $\varphi_n: A\to \M$ and $\psi_n: \M\to B$ be c.c.p. admissible maps with $\psi_n\circ\varphi_n\to \theta$ in point-norm topology. Extend $\varphi_n$ to u.c.p. admissible map $\tilde\varphi_n: \tilde A\to \M$ and note that $\psi_n\circ\tilde\varphi_n$ converges to $\tilde\theta$ in the point norm topology.
\end{proof}

 Let $E$ be a $C^*$-algebra and a left $\mathfrak A$-module with compatible actions, then $E$ is called $\mathfrak A$-{\it injective} if for every faithful representation $E\subseteq \mathbb B(H)$, where $H$ is a left  $\mathfrak A$-Hilbert space, there is an  $\mathfrak A$-conditional expectation $\mathbb E: \mathbb B(H)\to E$. The next lemma extends \cite[2.2.5]{bo}.

\begin{lemma} \label{in}
If $\mathfrak A$ is a unital $C^*$-algebra, $E$ is a unital  $C^*$-algebra and a unital   left $\mathfrak A$-injective module with compatible actions, $A$ is unital and $\tilde\varphi: A\to \mathbb M_n(E)$ is a c.p. admissible map, then there is a u.c.p. admissible map $\varphi: A\to \mathbb M_n(E)$ such that
$$\tilde\varphi(a)=\tilde\varphi(1_A)^{\frac{1}{2}}\varphi(a)\tilde\varphi(1_A)^{\frac{1}{2}},$$
for $a\in A$.
\end{lemma}
\begin{proof}
Consider a faithful embedding $E\subseteq \mathbb B(H)$ for a left  $\mathfrak A$-Hilbert space $H$ such that $E$ contains the identity of $\mathbb B(H)$. If $u:=\tilde\varphi(1_A)\in\mathbb M_n(E)\subseteq \mathbb M_n(\mathbb B(H))=\mathbb B(H^n)$ is invertible in $\mathbb B(H^n)$, we consider the u.c.p. admissible map $\varphi_1(a)=u^{-\frac{1}{2}}\tilde\varphi(a)u^{-\frac{1}{2}}$.

In general, let $p: H^n\to \ker\tilde\varphi(1_A)$ and put $p^\perp=1-p$, then for $0\leq a\leq 1_A$ and $\zeta\in\ker\tilde\varphi(1_A)$ we have
$$-\|\varphi(a)^{\frac{1}{2}}\zeta\|^2=\langle\big(\tilde\varphi(1_A)-\tilde\varphi(a)\big)\zeta,\zeta\rangle\geq 0,$$
thus $\zeta\in\ker\tilde\varphi(a)$, in particular, $p$ acts as identity on the image of $\tilde\varphi(a)$, that is, $\tilde\varphi(a)=p^\perp\tilde\varphi(a)=\tilde\varphi(a)p^\perp$, which then holds for any $a\in A$. We let $\mathfrak A$ act on $H$ by $\alpha\cdot \xi=( 1_E\cdot\alpha)(\xi),$ for $\alpha\in\mathfrak A$ and $\xi\in H$. This gives a canonical right module structure on $\mathbb B(H)$ such that $\tilde\varphi: A\to p^\perp \mathbb B(H^n)p^\perp$ is a c.p. admissible map. By the first part of the proof, there is a u.c.p. admissible map $\varphi_1: A\to p^\perp\mathbb B(H^n)p^\perp$ satisfying $u^{\frac{1}{2}}\varphi_1(a)u^{\frac{1}{2}}=p^\perp\tilde\varphi(a)p^\perp$, for $a\in A$. Put $\varphi_2(a):=\varphi_1(a)\oplus\omega(a)p$, for some c.p. admissible map $\omega: A\to \mathfrak A$, then $\varphi_2: A\to \mathbb B(H^n)$ is a u.c.p. admissible map.

By  $\mathfrak A$-injectivity of $E$, there is an  $\mathfrak A$-conditional expectation $\mathbb E: \mathbb B(H)\to E$. Let $\mathbb E_n: \mathbb B(H^n)\to \mathbb M_n(E)$ be the amplification of $\mathbb E$ and put $\varphi=\mathbb E_n\circ\varphi_2$. Then $\varphi: A\to \mathbb M_n(E)$ is a u.c.p. admissible map.
Since $p^\perp u=up^\perp=u$ and $pu=up=0$, we get
$$
u^{\frac{1}{2}}\varphi_2(a)u^{\frac{1}{2}}=u^{\frac{1}{2}}(\varphi_1(a)+\omega(a)p)u^{\frac{1}{2}}=p^\perp\tilde\varphi(a)p^\perp+0=\tilde\varphi(a),
$$
therefore,
$$
u^{\frac{1}{2}}\varphi(a)u^{\frac{1}{2}}=u^{\frac{1}{2}}\mathbb E_n(\varphi_2(a))u^{\frac{1}{2}}=\mathbb E_n(u^{\frac{1}{2}}\varphi_2(a)u^{\frac{1}{2}})=\mathbb E_n(\tilde\varphi(a))=\tilde\varphi(a),
$$
for $a\in A$.
\end{proof}

\begin{proposition} \label{inj1}
If $\mathfrak A$ is a unital $C^*$-algebra, $E$ is a unital  $C^*$-algebra and a neo-unital   right $\mathfrak A$-injective module with compatible actions, $\theta: A\to B$ is an $E$-nuclear module map, then there are u.c.p. admissible maps $\varphi_n: \tilde A\to \mathbb M_{k(n)}(E)$ and $\psi_n: \mathbb M_{k(n)}(E)\to B$ such that
$\psi_n\circ \varphi_n\to \tilde\theta$ in the point-norm topology.
\end{proposition}
\begin{proof}
Let $\tilde\varphi_n: A\to\M$ and $\tilde\psi_n:\M\to B$ be  c.c.p. admissible maps with $\tilde\psi_n\circ \tilde\varphi_n\to \theta$ in the point-norm topology. By Lemma \ref{in}, there are u.c.p. admissible maps $\varphi_n: A\to \M$ such that
$$\tilde\varphi_n(a)=\tilde\varphi_n(1_A)^{\frac{1}{2}}\varphi_n(a)\tilde\varphi_n(1_A)^{\frac{1}{2}},$$
for $a\in A$. For $x\in\M$, let
$$\psi_n(x)=\big(\tilde\psi_n (\tilde\varphi_n(1_A))\big)^{-\frac{1}{2}}\tilde\psi_n\big(\tilde\varphi_n(1_A)^{\frac{1}{2}}x\tilde\varphi_n(1_A)^{\frac{1}{2}}\big)\big(\tilde\psi_n( \tilde\varphi_n(1_A))\big)^{-\frac{1}{2}},$$
then $\psi_n$ is a u.c.p. admissible map, and  $\psi_n\circ \varphi_n$ converges to  $\tilde\theta$ in point-norm topology.
\end{proof}

\begin{lemma} \label{injc}
Let $\mathfrak A$ be a von Neumann algebra, $A$ be a unital $C^*$-algebra and a neo-unital right $\mathfrak A$-module, $B$ be a von Neumann algebra and a right $\mathfrak A$-module, and $E$ be a von Neumann algebra and  a unital  injective right $\mathfrak A$-module, all  with compatible actions. If $\theta: A\to B$ is an $E$-weakly unital nuclear module map, then there are u.c.p. admissible maps $\varphi_n: \tilde A\to \mathbb M_{k(n)}(E)$ and $\psi_n: \mathbb M_{k(n)}(E)\to B$ such that
$\psi_n\circ \varphi_n\to \theta$ in the point-ultraweak topology.
\end{lemma}
\begin{proof}
Let $\tilde\varphi_n: A\to\M$ and $\tilde\psi_n:\M\to B$ be  c.c.p. admissible maps with $\tilde\psi_n\circ \tilde\varphi_n\to \theta$ in the point-ultraweak topology. By Lemma \ref{in}, there are u.c.p. admissible maps $\varphi_n: A\to \M$ such that
$$\tilde\varphi_n(a)=\tilde\varphi_n(1_A)^{\frac{1}{2}}\varphi_n(a)\tilde\varphi_n(1_A)^{\frac{1}{2}},$$
for $a\in A$. Let $\rho_n:\M\to\mathfrak A$ be a c.p. module map, then  the map defined for $x\in\M$ by
$$\psi_n(x)=\big(1_B-\tilde\psi_n(\tilde\varphi_n(1_A))\big)\cdot\rho_n(x)+\tilde\psi_n\big(\tilde\varphi_n(1_A)^{\frac{1}{2}}x\tilde\varphi_n(1_A)^{\frac{1}{2}}\big),$$
is a u.c.p. admissible map, and  $\psi_n\circ \varphi_n$ converges to  $\theta$ in point-ultraweak topology.
\end{proof}

\begin{lemma} \label{wnucl}
Let $\mathfrak A$ be a unital $C^*$-algebra, $A$ be a unital $C^*$-algebra and a neo-unital right $\mathfrak A$-module, $B$ be a von Neumann algebra and a right $\mathfrak A$-module, and $E$ be a unital $C^*$-algebra and a neo-unital right $\mathfrak A$-module, all  with compatible actions. If $\theta: A\to B$ is a module map which could be approximated in point-ultraweak topology by $E$-factorable maps, then $\theta$ is $E$-weakly nuclear.
\end{lemma}
\begin{proof}
Take c.p. admissible maps $\varphi_n: A\to\M$ and $\psi_n: \M\to B$ such that $\psi_n\circ\varphi_n$ approximates $\theta$ in point-ultraweak topology. By \cite[3.8.1]{bo}, we may assume that the above sequence  approximates $\theta$ in point-SOT. By Lemma \ref{in}, we may assume that $\varphi_n$ is a u.c.p. admissible map. Then $b_n:=\psi_n(1_A)\to 1_B$ in SOT. For $\delta>0$, take $p_n:=\chi_{[0, 1+\delta]}(b_n)\in B$ and observe that $p_n(b_n-1_B)\to 0$ in SOT, and so $1-p_n\leq \frac{1}{\delta}|1-b_n|\to 0$ in SOT. Hence $p_n\psi_n\circ\varphi_n(\cdot)\to\theta$ in SOT. Take $\psi_n^{'}=p_n\psi_n(\cdot)p_n$, which is an admissible map. Then $\|\psi_n^{'}\|\leq 1+\delta$ and we only need to show that $\psi_n^{'}\circ\varphi_n$ approximates $\theta$ in point-SOT, but this could be done exactly as in the proof of \cite[3.8.2]{bo}.  
\end{proof}

In the next theorem, we give $A^{**}$ the canonical right module structure via,
$$\langle\alpha\cdot x^{*}, x\rangle=\langle x^{*},  x\cdot\alpha\rangle, \ \  \langle   x^{**}\cdot\alpha, x^{*}\rangle=\langle x^{**}, \alpha\cdot x^*\rangle,$$
for $\alpha\in\mathfrak A, x\in A, x^*\in A^*,$ and $x^{**}\in A^{**}.$

\begin{theorem}\label{sd}
Let $\mathfrak A$ be a von Neumann algebra and  $A$ be a unital $C^*$-algebra and a neo-unital left $\mathfrak A$-module with compatible actions. If $A^{**}$ is $\mathfrak A$-semidiscrete then $A$ is $\mathfrak A$-nuclear.
\end{theorem}
\begin{proof}
By Lemma \ref{convex}, we need to show that, given $\varepsilon>0$ and finite subsets $\mathfrak F\subseteq A$ and $\mathfrak S\subseteq A^*$, there is a positive integer $n$ and c.c.p. module maps $\varphi: A\to\mathbb M_n(\mathfrak A)$ and  $\psi: \mathbb M_n(\mathfrak A)\to A$ with
$$|\phi(\psi\circ\varphi(a))-\phi(a)|<\varepsilon,\ \ (a\in \mathfrak F, \phi\in\mathfrak S).$$
By Lemma \ref{injc}, One could find u.c.p. admissible maps $\varphi^{'}: A^{**}\to\mathbb M_n(\mathfrak A)$ and  $\psi^{'}: \mathbb M_n(\mathfrak A)\to A^{**}$ with
$$|\phi(\psi^{'}\circ\varphi^{'}(a))-\phi(a)|<\varepsilon,\ \ (a\in \mathfrak F, \phi\in\mathfrak S).$$
By \cite[Lemma 3]{am}, one can associate to $\psi^{'}$  some $\hat\psi^{'}\in \mathbb M_n(A^{**})_{+}$, in which $\mathbb M_n(A)_{+}$ is ultraweakly dense, and so there is a net of c.p. admissible maps $\psi_\lambda: \mathbb M_n(\mathfrak A)\to A$ with $\psi_\lambda\to \psi^{'}$ in point-ultraweak topology. Since $\psi_\lambda(1_{\mathbb M_n(\mathfrak A)})\to 1_A$ in the weak topology of $A$, replacing $\psi^{'}$ with an appropriate convex combination, we get a c.p. admissible map $\psi^{''}: \mathbb M_n(\mathfrak A)\to A$ with
$$\|1_A-\psi^{''}(1_{\mathbb M_n(\mathfrak A)})\|<\varepsilon, \ |\phi(\psi\circ\varphi(a))-\phi(a)|<\varepsilon,\ \ (a\in \mathfrak F, \phi\in\mathfrak S).$$
Put $a=\psi^{''}(1_{\mathbb M_n(\mathfrak A)})$, then we have the claimed inequality for the restriction $\varphi$ of $\varphi^{'}$ to $A$ and $\psi=\frac{1}{\|a\|}\psi^{''}$.
\end{proof}

\begin{proposition} \label{inj}
Let $\mathfrak A$ be an injective $C^*$-algebra.

$(i)$ If $A$ is $\mathfrak A^{**}$-nuclear, then $A$ is $\mathfrak A$-nuclear.

$(ii)$ If  $A^{**}$ is $\mathfrak A^{**}$-semidiscrete then $A$ is $\mathfrak A$-nuclear.
\end{proposition}
\begin{proof}
$(i)$ Let $\tilde\varphi_n: A\to\MMM$ and $\tilde\psi_n:\MMM\to A$ be  c.c.p. admissible maps with $\tilde\psi_n\circ \tilde\varphi_n\to id_A$ in the point-norm topology. By a modification of \cite[Lemma 4]{am}, we may write $\tilde\varphi_n\in CP_{\mathfrak A}(\mathbb M_{k_n}(A), \mathfrak A^{**})$ and $\tilde\psi_n\in CP_{\mathfrak A}(\mathfrak A^{**}, \mathbb M_{k_n}(A))$. Let $\psi_n$ be the restriction of $\tilde\psi_n$ to $\mathfrak A$. By (a one-sided version of) \cite[Theorem 3.2]{fp}, $\mathfrak A$ is injective as an operator $\mathfrak A$-module. Let $\mathbb E:\mathfrak A^{**}\to \mathfrak A$ be an $\mathfrak A$-conditional expectation, which exist by $\mathfrak A$-injectivity (cf., \cite[Definition 1]{fp}), and put $\varphi_n=\mathbb E\circ\tilde\varphi_n$, then $\varphi_n: A\to\MM$ and $\psi_n:\MM\to A$ are c.c.p. admissible maps with $\psi_n\circ \varphi_n\to id_A$ in the point-norm topology.

$(ii)$ This does not follow from part $(i)$ and the above theorem (since $A$ and $A^{**}$ are not $\mathfrak A^{**}$-modules). However, using the idea of a conditional expectation we may rewrite the proof of Theorem \ref{sd} to prove $\mathfrak A$-nuclearity of $A$.
\end{proof}

\begin{remark} In part $(ii)$ of the above proposition, we could conclude that $A$ is $\mathfrak A$-nuclear when the canonical inclusion $\iota: A\hookrightarrow A^{**}$ is $\mathfrak A$-weakly nuclear (with almost the same proof.)
\end{remark}

We prove the converses of Theorem \ref{sd} and Proposition \ref{inj}$(ii)$ in a subsequent section. Next let us give some classes of examples of nuclear and semidiscrete $\mathfrak A$-modules.

We say that $A$ is {\it locally} $E$-{\it nuclear} if for each finite subset $\mathfrak F\subseteq A$ and $\varepsilon>0$, there is an $E$-nuclear $C^*$-subalgebra and  submodule $B\subseteq A$ which contains $\mathfrak F$ within $\varepsilon$ in norm.

\begin{lemma} \label{lnuc}
If $E$ is $\mathfrak A$-injective,  local $E$-nuclearity is equivalent to $E$-nuclearity.
\end{lemma}
\begin{proof}
If $A$ is locally $E$-nuclear, given finite subsets $\mathfrak F\subseteq A$ and $\mathfrak S\subseteq A^*$,  $\varepsilon>0$, and $B$ as above, let  $\mathfrak G\subseteq B$ and $\mathfrak T\subseteq B^*$ be within  $\varepsilon$ in norm. Then there is a positive integer $n$ and c.c.p. admissible maps $\varphi: B\to \mathbb M_n(E)$ and $\psi: \mathbb M_n(E)\to B$ with
$$|\eta(\psi\circ\varphi(b))-\eta(b)|<\varepsilon,\ \ (b\in \mathfrak G, \eta\in\mathfrak T).$$
Consider $\psi$ as a map into $A$ and extend $\varphi $ to a c.c.p. admissible map $\tilde\varphi$ on $A$ using the injectivity of $\mathbb M_n(E)$ \cite[Corollary 3.9]{am}. Then
$$|\xi(\psi\circ\tilde\varphi(a))-\xi(a)|<3\varepsilon,\ \ (a\in \mathfrak F, \xi\in\mathfrak S).$$
\end{proof}
If there is an inductive system of algebras of the form
$$\mathbb M_{n_1}(\mathfrak A)\oplus\cdots\mathbb M_{n_k}(\mathfrak A)$$
with injective connecting module maps, the corresponding inductive limit (with its canonical $\mathfrak A$-module structure) is called an $\mathfrak A$-AF-algebra. When $\mathfrak A$ is an injective $C^*$-algebra (and so an injective $\mathfrak A$-module), it follows from the above lemma (and the fact that $\mathbb M_n(\mathfrak A)$ is $\mathfrak A$-nuclear) that $\mathfrak A$-AF-algebras are $\mathfrak A$-nuclear. As the second example, consider a nuclear $C^*$-algebra $B$ and let $\mathfrak A$ act on $A=\mathfrak A\otimes B$ from right by $(\beta\otimes b)\cdot\alpha=\beta\alpha\otimes b$, then it is easy to see that $A$ is $\mathfrak A$-nuclear. As a more concrete example, in the next section we shall examine the nuclearity of $C^*$-algebras of inverse semigroups as modules on the $C^*$-algebra of their subsemigroup of idempotents.

When $\mathfrak A$ is a von Neumann algebra and $B$ is a semidiscrete von Neumann algebra (in particular, if $B$ is of Type I)  then $A=\mathfrak A\bar\otimes B$ is $\mathfrak A$-semidiscrete. As another classes of examples, let us show that certain ``Type I'' modules are $\mathfrak A$-semidiscrete.

\begin{lemma} \label{sot}
If $X$ is  a left Hilbert $\mathfrak A$-module and $A$ is a von Neumann algebra and a $\mathfrak A$-module with compatible action, and there is a faithful representation $A\subseteq \mathbb B(X)$, preserving the module actions, and there is a net of projections $p_i\in A$ such that $p_i\to 1_X$ in the strong operator topology and each corner $p_iAp_i$ is  $\mathfrak A$-semidiscrete, then so is $A$.
\end{lemma}
\begin{proof}
Given finite subsets $\mathfrak F\subseteq A$ and $\mathfrak S\subseteq A_{*}$ and $\varepsilon>0$, we need to find a positive integer $n$ and c.c.p. admissible maps $\varphi: A\to \mathbb M_n(\mathfrak A)$ and $\psi: \mathbb M_n(\mathfrak A)\to A$ with
$$|\eta(\psi\circ\varphi(a))-\eta(a)|<\varepsilon,\ \ (a\in \mathfrak F, \eta\in\mathfrak S).$$
Choose a projection $p\in M$ such that $pMp$ is $\mathfrak A$-semidiscrete and $\eta(1-p)<\varepsilon/4$, for each $\eta\in\mathfrak S$. By assumption, there are  c.c.p. admissible maps $\tilde\varphi: pAp\to \mathbb M_n(\mathfrak A)$ and $\psi: \mathbb M_n(\mathfrak A)\to A$ with
$$|\eta(\psi\circ\tilde\varphi(pap))-\eta(pap)|<\varepsilon/4,\ \ (a\in \mathfrak F, \eta\in\mathfrak S).$$
The map $\varphi(a):=\tilde\varphi(pap)$ along with $\psi$ have the required property.
\end{proof}

In the above situation, if $X=H\otimes \mathfrak A$, for some Hilbert space $H$, then by Lemma \ref{fr}, there is a net $\{q_i\}$ of finite rank projections in $\mathbb B(H\otimes \mathfrak A)$, converging to the identity in the strong operator topology and each corner $q_i\mathbb B(H\otimes\mathfrak A)q_i=\mathbb M_{n_i}(\mathfrak A)$, which is always $\mathfrak A$-semidiscrete.  Hence by the above lemma, $\mathbb B(H\otimes \mathfrak A)$ is $\mathfrak A$-semidiscrete. A little effort then shows that ``type I'' modules (with trivial action on the first leg) of the form $$\prod_{i\in I} B_i\bar\otimes \mathbb B(H_i\otimes \mathfrak A)$$
are $\mathfrak A$-semidiscrete where $B_i$ is an abelian von Neumann algebra and $H_i$ is a Hilbert space of dimension $i$. Note that these are not of Type I as a von neumann algebra (if $H=\mathbb C$ and $\mathfrak A$ is unital, then $\mathbb B(H\otimes \mathfrak A)=\mathfrak A$), but they resemble type I objects in the category of $\mathfrak A$-modules.     A similar argument shows that
$$\bigoplus_{i\in I} B_i\otimes \mathbb K(H_i\otimes \mathfrak A)$$
is $\mathfrak A$-nuclear, for commutative $C^*$-algebras $B_i$ and Hilbert spaces $H_i$.

\section{tensor products}

In this section we give tensor product characterizations of $\mathfrak A$-nuclearity and $\mathfrak A$-exactness. First we need to introduce the module versions of minimal and maximal tensor products. 

Let $A$ and $B$ be $C^*$-algebras and right $\mathfrak A$-modules with compatible actions. Let $\pi: A\to\mathbb B(H\otimes\mathfrak A)$ and $\sigma: B\to\mathbb B(K\otimes\mathfrak A)$  be  faithful representations, for Hilbert spaces $H$ and $K$, which are also right module maps.  For left Hilbert modules $H_\mathfrak A:=H\otimes\mathfrak A$ and $K_\mathfrak A:=K\otimes\mathfrak A$, and  the  Hilbert space  $\bar H$ of the corresponding contragredient representation, we may regard the left $\mathfrak A^{op}$-module $\bar H_\mathfrak A:=\bar H\otimes\mathfrak A^{op}$ as a right $\mathfrak A$-module via  
$$(\bar\xi\otimes \beta^{op})\cdot \alpha:= \bar\xi\otimes (\alpha^{op}\beta^{op})=\bar\xi\otimes (\beta\alpha)^{op}, \ \ (\xi\in H, \alpha,\beta\in\mathfrak A),$$ 
which in turn induces a $*$-homomorphism 
$$\phi: \mathfrak A\to \mathbb B(\bar H_\mathfrak A); \ \phi(\alpha)(\bar\xi\otimes \beta^{op}):=\bar\xi\otimes (\beta\alpha)^{op}.$$ Consider the interior tensor product $\bar H_\mathfrak A\otimes_{\mathfrak A} K_\mathfrak A$ \cite[page 40]{lan} (where here we use this notation for the completion, denoted in \cite{lan} by $\bar H_\mathfrak A\otimes_{\phi} K_\mathfrak A$). Similar to what is observed by Lance \cite[page 42]{lan}, $\bar H_\mathfrak A\otimes_{\mathfrak A} Y$ is canonically identified with $\bar H\otimes Y$, for each left Hilbert $\mathfrak A$-module $Y$. In particular, $\bar H_\mathfrak A\otimes_{\mathfrak A} K_\mathfrak A\simeq \bar H\otimes K\otimes \mathfrak A$, canonically.  

Let $A\odot B^{op}$ be the algebraic tensor product of $A$ with the opposite algebra of $B$ and consider the faithful representation $\pi\otimes \sigma^{op}:  A\odot B^{op}\to \mathbb B(H_{\mathfrak A})\odot\mathbb B(K_{\mathfrak A})^{op}$ given by $\pi\otimes \sigma^{op}(a\otimes b^{op})=\pi(a)\otimes \sigma(b)^{op}$. Let $I_{\mathfrak A}$ be the two sided ideal of $A\odot B^{op}$ generated by elements of the form
$$a\cdot\alpha\otimes b^{op}-a\otimes \alpha\cdot_{op}b^{op}=a\cdot\alpha\otimes b^{op}-a\otimes (b\cdot\alpha)^{op},$$
for $\alpha\in\mathfrak A, a\in A, b\in B$ and form the corresponding quotient $A\odot_{\mathfrak A}B^{op}:=(A\odot B^{op})/I_{\mathfrak A}$. Let us do the same thing for algebras of adjointable operators and write $\mathbb B(H_{\mathfrak A})\odot_{\mathfrak A}\mathbb B(K_{\mathfrak A})^{op}:=(\mathbb B(H_{\mathfrak A})\odot\mathbb B(K_{\mathfrak A})^{op})/I$ for the corresponding ideal $I$. Now since $\pi\otimes \sigma^{op}(I_{\mathfrak A})= I\cap \ \text{range}(\pi\otimes \sigma^{op})$, we could lift $\pi\otimes \sigma^{op}$ to a faithful representation $$\pi\otimes \sigma^{op}: A\odot_{\mathfrak A}B^{op}\to \mathbb B(H_{\mathfrak A})\odot_{\mathfrak A}\mathbb B(K_{\mathfrak A})^{op}\subseteq \mathbb B(\bar H_{\mathfrak A}\otimes_{\mathfrak A} K_{\mathfrak A}),$$
where the last inclusion is proved as in \cite[3.3.9]{bo}.
For a finite sum $u=\sum_i a_i\otimes b_i\in A\odot_{\mathfrak A}B^{op}$ let us define the minimal norm of $u$ by
$$\big\|\sum_i a_i\otimes b_i\big\|_{min}=\big\|\sum_i\pi(a_i)\otimes \sigma^{op}(b_i)\big\|,$$
where the norm on the right hand side is in $\mathbb B(\bar H_{\mathfrak A}\otimes_{\mathfrak A} K_{\mathfrak A})=\mathbb B(\bar H\otimes K\otimes {\mathfrak A})$. This is a well defined $C^*$-norm as we work with faithful representations, and the completion $A\otimes_{\mathfrak A}^{min}B^{op}$ of $A\odot_{\mathfrak A}B^{op}$ is a $C^*$-algebra and a right $\mathfrak A$-module, and we could extend $\pi\otimes \sigma^{op}$ to a faithful representation $\pi\otimes \sigma^{op}: A\otimes_{\mathfrak A}^{min}B^{op}\to \mathbb B(\bar H\otimes K\otimes {\mathfrak A}).$ Moreover the right action of $\mathfrak A$ on $A\odot_{\mathfrak A}B^{op}$ is continuous in this norm and extends to a right action on $A\otimes_{\mathfrak A}^{min}B^{op}$ and the last representation is also a right module map.

Next we check the independence of the $min$ norm from the faithful representations involved. 

\begin{lemma} \label{min}
The min-norm is independent of the choice of the faithful representations $\pi: A\to\mathbb B(H_{\mathfrak A})$ and $\sigma: B\to\mathbb B(K_{\mathfrak A})$ in the left $\mathfrak A$-Hilbert modules $H_{\mathfrak A}$ and $K_{\mathfrak A}$.
\end{lemma}
\begin{proof}
We need to check that changing $\sigma$ to $\sigma^{'}: B\to\mathbb B(K^{'}_{\mathfrak A})$ does not effect the norm. By Lemma \ref{fr}, choose an increasing net $\{q_i\}$ of finite rank (say $n_i$) projections in $\mathbb B(H_{\mathfrak A})$ converging to the identity in the strong operator topology. Then $p_i\mathbb B(H_{\mathfrak A})p_i\cong\mathbb M_{n_i}(\mathfrak A)$ and there is a unique $C^*$-norm on $\mathbb M_{n_i}(\mathfrak A)\odot_{\mathfrak A}B^{op}=\mathbb M_{n_i}(B^{op})$, thus
$$\big\|\sum_i(p_i\pi(a_i)p_i)\otimes \sigma^{op}(b_i)\big\|=\big\|\sum_i(p_i\pi(a_i)p_i)\otimes \sigma^{' op}(b_i)\big\|,$$
and taking supremum over $i$, we get the result.
\end{proof}

Similarly, we may define the maximal norm of $u=\sum_i a_i\otimes b_i\in A\odot_{\mathfrak A}B^{op}$ by
$$\big\|\sum_i a_i\otimes b_i\big\|_{max}=\sup\{\big\|\Pi(\sum_i a_i\otimes b_i)\big\|: \ \Pi\in Hom_{\mathfrak A}(A\odot_{\mathfrak A}B^{op}, \mathbb B(X))\},$$
where the supremum ranges over all left $\mathfrak A$-Hilbert modules $X$ and modules maps and $*$-homomorphisms $\Pi$. This is a $C^*$-norm  and the completion $A\otimes_{\mathfrak A}^{max}B^{op}$ of $A\odot_{\mathfrak A}B^{op}$ is a $C^*$-algebra and a right $\mathfrak A$-module, and for each $C^*$-algebra and right $\mathfrak A$-module $C$ with compatible actions, we could extend each $\Pi\in Hom_{\mathfrak A}(A\odot_{\mathfrak A}B^{op}, C)$ to a (bounded) $*$-homomorphism and module map $\tilde\Pi\in Hom_{\mathfrak A}(A\otimes_{\mathfrak A}^{max}B^{op}, C)$. When $\mathfrak A$ is a von Neumann algebra and $A$ and $B$ be von Neumann algebras and right $\mathfrak A$-modules with compatible actions, we may construct the von Neumann algebra and right $\mathfrak A$-module $A\bar\otimes_{\mathfrak A}B^{op}$ with a similar universal property for each von Neumann algebra and right $\mathfrak A$-module $C$.

Before we treat the non unital case, we need to take care of the possibility of representing faithfully on Hilbert modules of the form $H\otimes{\mathfrak A}$, for a Hilbert space $H$. 
Indeed, this is not always possible and we shall assume this whenever we work with the  minimal module tensor product. The good news is that such a faithful (non degenerate) representation exists by the Kasparov absorption theorem \cite[Theorem 6.2]{lan}, when $A$ is separable (see, \cite[Lemma 6.4]{lan}). For the rest of this section, to guarantee the existence of a faithful representation as above, one may  assume that the $C^*$-algebra involved is separable. When this is not possible (like in the case of infinite dimensional von Neumann algebras), we restrict ourselves to modules which have such a faithful representation.

To treat the non unital modules, one should note that when $\mathfrak A$ is a unital $C^*$-algebra, as in \cite[3.3.12]{bo}, for any $C^*$-norm $\varrho$ on $A\odot_{\mathfrak A}B^{op}$ can be extended to a $C^*$-norm on $(A\oplus \mathfrak A)\odot_{\mathfrak A}(B\oplus \mathfrak A)^{op}$ and $(A\oplus \mathfrak A)\otimes_{\mathfrak A}^{\varrho}(B\oplus \mathfrak A)^{op}$ is a unital $C^*$-algebra and a neo-unital module.

The following proposition follows from the independence of the {\it min} norm from the choice of faithful representations and the universal property of the {\it max} norm. Note that here the right hand sides of the isomorphisms are not $\mathfrak A$-modules, as the module action on $A\odot B^{op}$ is not continuous in the min or max $C^*$-norms in general.

\begin{proposition} \label{iso}
With the above notations, $A\otimes_{\mathfrak A}^{min}B^{op}\cong (A\otimes_{min}B^{op})/\overline{I_{\mathfrak A}}^{min}$ and $A\otimes_{\mathfrak A}^{max}B^{op}\cong (A\otimes_{max}B^{op})/\overline{I_{\mathfrak A}}^{max}$, as $C^*$-algebras.
\end{proposition}

The next lemma is proved similar to the classical case \cite[IV.3.1.1]{b}.

\begin{lemma} \label{tensor}
If $A$ and $B$ are $C^*$-algebras and $\mathfrak A$-modules with compatible actions, then $A\otimes_{\mathfrak A}^{min}B$ is $\mathfrak A$-nuclear iff both $A$ and $B$ are $\mathfrak A$-nuclear.
\end{lemma}

Next, let us extend a result due to Takesaki \cite{t} to the set up of $C^*$-modules.
We say that a $C^*$-norm $\tilde\varrho$ on $A\odot_{\mathfrak A}B^{op}$ is a module norm if the module action of $\mathfrak A$ on $A\odot_{\mathfrak A}B^{op}$ is $\tilde\varrho$-continuous. In this case, $\tilde\varrho$ induces a $C^*$-norm $\varrho$ on $A\odot B^{op}$ as follows: For  $u=\sum_i a_i\otimes b_i\in A\odot B^{op}$, define
$$\varrho(u):=\sup\{\|\pi(u)\|: \ \pi(I_{\mathfrak A})=0, \ \tilde\pi=\tilde\varrho\text{-continuous module map}\},$$
where the supremum ranges over all representations $\pi:A\odot B^{op}\to \mathbb B(X)$ in a Hilbert $\mathfrak A$-module $X$, and $\tilde\pi(u+I_{\mathfrak A}):=\pi(u)$. After completion, $A\otimes_{\mathfrak A}^{\tilde\varrho}B^{op}\cong (A\otimes_{\varrho}B^{op})/\overline{I_{\mathfrak A}}^{\varrho}$. It follows from the above discussion that $min$ and $max$ are modules norms. The module version of the Takesaki theorem now asserts that $min$ and $max$ are minimum and maximum among module norms (like the $min$ and $max$ $C^*$-norms that are minimum and maximum among all $C^*$-norms).

\begin{proposition}[Takesaki] \label{tak}
For each module norm  $\tilde\varrho$ on $A\odot_{\mathfrak A}B^{op}$, $\|\cdot\|_{min}\leq\tilde\varrho(\cdot)\leq\|\cdot\|_{max}.$ In particular, $\tilde\varrho$ is a cross norm and there are canonical surjective $*$-homomorphisms and module maps $A\otimes_{\mathfrak A}^{max}B^{op}\to A\otimes_{\mathfrak A}^{\tilde\varrho}B^{op}\to A\otimes_{\mathfrak A}^{min}B^{op}$.
\end{proposition}
\begin{proof} The second inequality is clear. For the first, given $u\in A\odot B^{op}$, from Proposition \ref{iso}, we have
$$\|u+I_{\mathfrak A}\|_{min}\leq \|u\|_{min}\leq\varrho(u)=\sup_{\pi}\|\pi(u)\|=\sup_{\pi}\|\tilde\pi(u+I_{\mathfrak A})\|\leq\tilde\varrho(u+I_{\mathfrak A}),$$
where both supremums run over the set of all representations $\pi$ used in the definition of $\varrho(u)$.
\end{proof}

Next, we could handle the problem of continuity of the maps on module tensor products using the module version of the Kasparov-Stinespring dilation theorem \cite[Theorem 5.6]{lan}. Recall that a c.p. map $\phi: A\to \mathbb B(X)$ is called {\it strict} if $(\phi(e_i))$ is strictly Cauchy, for some bounded approximate identity $(e_i)$ of $A$. Non degenerate c.p. maps are automatically strict and u.c.p. maps are always non degenerate \cite[page 49]{lan}. 

\begin{proposition}\label{cont}
Let $A, B, C,$ and $D$ be $C^*$-algebras and right $\mathfrak A$-modules with compatible actions, let $\varphi: A\to C$ and $\psi:B\to D$ be strict c.p. module maps. Let $I_{\mathfrak A}\subseteq A\odot B^{op}$ and  $I_{\mathfrak A}^{'}\subseteq C\odot D^{op}$ be the ideals whose quotients give the associated module tensor products $A\odot_{\mathfrak A} B^{op}$ and $C\odot_{\mathfrak A} D^{op}$. Then $\varphi\otimes\psi(I_{\mathfrak A})\subseteq I_{\mathfrak A}^{'}$ and so we have a c.p. lift $\varphi\otimes\psi^{op}: A\odot_{\mathfrak A} B^{op}\to C\odot_{\mathfrak A} D^{op}$. This is continuous if both domain and range are endowed with the min or max module norms, and for the corresponding extensions we have
$$\|\varphi\otimes_{max}\psi^{op}\|=\|\varphi\otimes_{min}\psi^{op}\|\leq\|\varphi\|\|\psi\|.$$
In particular, when $\varphi$ and $\psi$ are strict c.c.p. admissible maps and $\varphi_i: A\to C$ are strict c.c.p. admissible maps converging to $\varphi$ in point-norm topology, then $\varphi_i\otimes_{min}\psi^{op}\to\varphi\otimes_{min}\psi^{op}$ in the point-norm topology, and the same holds for the $max$ norm.
\end{proposition}
\begin{proof} Consider faithful representations $C\subseteq \mathbb B(X)$ and  $D\subseteq \mathbb B(Y)$, for left Hilbert $\mathfrak A$-modules $X$ and $Y$, and let $\pi_A: A\to \mathbb B(X_\varphi)$ and $\pi_B: B\to \mathbb B(Y_\psi)$ be the corresponding Kasparov-Stinespring dilations with associated bounded operators $v_\varphi\in \mathbb B(X_\varphi)$ and $v_\psi\in \mathbb B(Y_\psi)$. A close inspection of the proof of KSGNS construction \cite[pages 50-51]{lan} shows that the representations $\pi_A$ and $\pi_B$ and operators $v_\varphi$, $v_\psi$ could also be chosed to be  module maps. By an argument similar to that used in the definition of the module min norm (see beginning of this section), we get a $*$-homomorphism and module map
	$$\pi_A\otimes \pi_B^{op}: A\otimes^{min}_\mathfrak A B^{op}\to \mathbb B(\bar X_\varphi\otimes_{\mathfrak A}Y_\psi), $$  
where $\bar X_\varphi$ is the corresponding contragradient right Hilbert $\mathfrak A$-module. Put 
$$ \varphi\otimes_{min}\psi^{op}:=(v_\varphi\otimes v_\psi)^*(\pi_A\otimes \pi_B^{op}(\cdot)(v_\varphi\otimes v_\psi), $$ 
then, since $\pi_A, \pi_B$ and $v_\varphi, v_\psi$ are module maps, $ \varphi\otimes_{min}\psi^{op}(I_mathfrak A)\subseteq I^{'}_\mathfrak A$. Finally, 
$$\|\varphi\otimes_{min}\psi^{op}\|\leq\|v_\varphi\otimes v_\psi\|^2=\|v_\varphi\|^2\|v_\psi\|^2=\|\varphi\|\|\psi\|.$$

For the max norm, fist consider the special case that the algebras involved are unital, $B=D$, and $\psi=$id$_{B}$. Consider a faithful representation $C\otimes^{max}_\mathfrak A B\subseteq \mathbb B(X)$, and regards $\varphi$ as an strict c.p. map $\varphi: A \to \mathbb B(X)$. Identifying $B$ and $C$ with $1_C\otimes B$ and $C\otimes 1_B$, we may further assume that $B\subseteq \varphi(A)^{'}\subseteq \mathbb B(X)$. Let $\pi_A: A\to \mathbb B(X_\varphi)$ be the corresponding Kasparov-Stinespring dilation of $\varphi$ and note that $\pi_A$ and id$_B$ have commuting ranges. Use universality to get $\pi_A\otimes{\rm id}^{op}: A\otimes_{max} B^{op}\to  C\otimes_{max} B^{op}$, and lift it to a map (again denoted by the same notation) between the corresponding module max tensor products,  put
$$\varphi\otimes_{max} {\rm id}_B^{op}:= (v_\varphi\otimes {\rm id})^{*}(\pi_A\otimes{\rm id}^{op})(\cdot)(v_\varphi\otimes {\rm id}),$$
and argue as in the min case to get the inequality for the norms. Reduce the general unital case to the above case by putting $\varphi\otimes_{max} \psi^{op}:=(\varphi\otimes_{max} {\rm id}_D^{op})({\rm id}_A\otimes\psi^{op})$. Finally, reduce to the unital case by Proposition \ref{unital} and the fact that u.c.p. maps are automatically strict c.p.
\end{proof}

Note that the above result holds also for c.b. module maps for the minimal norm (but not for the maximal norm), and it fails for bounded or positive module maps. Let $tr_n: \mathbb M_n(\mathfrak A)\to \mathbb M_n(\mathfrak A); [\alpha_{ij}]\mapsto [\alpha_ij]^{tr}:=[\alpha_{ji}^{*}]$ be the transpose map and $id_n: : \mathbb M_n(\mathfrak A)\to \mathbb M_n(\mathfrak A)$ be the identity map. Then
$tr_n$ is a positive isometry, and it is unital when $\mathfrak A$ is unital. We have $\|tr_n\|=1$ and $\|tr_n\otimes id_n\|\geq n$ (for the last inequality, assume that $\mathfrak A$ is unital. Let $\{e_{ij}\}$ be a system of matrix units for $\mathbb M_n(\mathfrak A)$, where $e_{ij}$ is a matrix with $1_{\mathfrak A}$ at the $(i,j)$-th position and $0$ elsewhere. Consider the unitary element $u=\sum_{i,j} e_{ij}\otimes e_{ji}\in  \mathbb M_n(\mathfrak A)\odot\mathbb M_n(\mathfrak A)$, then $(tr_n\otimes id_n)(u)=\sum_{i,j} e_{ji}\otimes e_{ji}$ and $\frac{1}{n}\sum_{i,j} e_{ji}\otimes e_{ji}$ is a projection in $ \mathbb M_n(\mathfrak A)\odot\mathbb M_n(\mathfrak A)$.) Now for $A=\mathbb K(\ell^2\otimes \mathfrak A)\oplus \mathfrak A$, the transpose map $tr: A\to A$ is a positive isometry with $tr\otimes id: A\otimes_{\mathfrak A}^{min} A\to A\otimes_{\mathfrak A}^{min} A$ unbounded (unless $\mathfrak A=0!$)

Next, we turn to the problem of inclusions. This is handled rather easily for the $min$ norm: if $A\subseteq B$ and $C\subseteq D$ are $C^*$-subalgebras and submodules, we have $A\odot C^{op}\subseteq B\odot D^{op}$ and the same holds for the corresponding ideals defining the algebraic module tensor products, hence $A\odot_{\mathfrak A} C^{op}\subseteq B\odot_{\mathfrak A} D^{op}$. Now by Lemma \ref{min} (and the remarks before that) the restriction of the $min$ norm on the bigger algebra to the smaller one is the same as the original $min$ norm on the latter, therefore, $A\otimes_{\mathfrak A}^{min} C^{op}\subseteq B\otimes_{\mathfrak A}^{min} D^{op}$. For the $max$ norm such a general inclusion result fails, however we have it under certain natural extension conditions. We say that the inclusion $A\subseteq B$ is $\mathfrak A$-{\it extendable} if for every  Hilbert space $H$ and each non degenerate $*$-homomorphism and module map $\pi: A\to \mathbb B(H\otimes \mathfrak A)$, there is a c.c.p. admissible map $\varphi: B\to \pi(A)^{''}$ extending $\pi$. Since $\pi$ could always be replaced by a faithful representation with the same range as $\pi$, it is enough to check  $\mathfrak A$-extendability it is enough to check the above condition just for faithful representations.    

\begin{proposition}\label{incl}
Let $B$ and $D$ be separable $C^*$-algebras and right $\mathfrak A$-modules with compatible actions, and $A\subseteq B$ and $C\subseteq D$ are $C^*$-subalgebras and submodules. Then

$(i)$ \ $A\otimes_{\mathfrak A}^{min} C^{op}\subseteq B\otimes_{\mathfrak A}^{min} D^{op}$,

$(ii)$ \ $A\otimes_{\mathfrak A}^{max} C^{op}\subseteq B\otimes_{\mathfrak A}^{max} D^{op}$, whenever both inclusions $A\subseteq B$ and $C\subseteq D$ are $\mathfrak A$-extendable.
\end{proposition}
\begin{proof} Part $(i)$ is checked in the above paragraph. We only need to prove $(ii)$ for the case $C=D$. By universality, there is a canonical $*$-homomorphism and module map  $\psi:  A\otimes_{\mathfrak A}^{max} C^{op}\to B\otimes_{\mathfrak A}^{max} C^{op}$; we need to check that $\psi$ is faithful. Let $q: A\odot C^{op}\to A\odot_{\mathfrak A} C^{op}=(A\odot C^{op})/ I_{\mathfrak A}$ be the quotient map. Take a faithful non degenerate representation and module map $\pi: A\otimes_{\mathfrak A}^{max} C^{op}\to \mathbb B(H\otimes \mathfrak A)$, for some  Hilbert  space  $H$, and as in \cite[3.2.6]{bo}, find representations $\pi_A: A\to \mathbb B(H_\mathfrak A)$ and  $\pi_C: C\to \mathbb B(H_\mathfrak A)$ with $\pi\circ q=\pi_A\times \pi_C^{op}$ and observe that these representations could be chosen to be module maps as well (by a careful inspection of the proof of the last citation). Now $\pi_C(C)\subseteq \pi_A(A)^{'}$ and since $(\pi_A\times \pi_C^{op})(I_{\mathfrak A})=0$, we have
$$ \pi_A(a\cdot\alpha) \pi_C(c)=\pi_A(a) \pi_C(c\cdot\alpha)\ \ (\alpha\in \mathfrak A, a\in A, c\in C).$$
The commuting inclusions $\pi_C(C)\hookrightarrow \mathbb B(H_\mathfrak A)$ and $\pi_A(A)^{''}\hookrightarrow \mathbb B(H_\mathfrak A)$ give a representation and module map  $\Pi: \pi_A(A)^{''}\odot \pi_C^{op}(C^{op})\to \mathbb B(H_\mathfrak A)$. The above relation between $\pi_A$ and $\pi_C$ shows that $\Pi$ vanishes on $I_{\mathfrak A}$ and so, by universality, it extends to a representation and module map  $\tilde\Pi: \pi_A(A)^{''}\otimes_{\mathfrak A}^{max} \pi_C^{op}(C^{op})\to \mathbb B(H_\mathfrak A)$. Finally extend $\pi_A$ to a c.c.p. admissible map $\varphi: B\to \pi_A(A)^{''}$ and observe that $\tilde\Pi\circ(\varphi\otimes_{max}\pi_C^{op})\circ\psi=\pi$. Therefore, $\psi$ is faithful as $\pi$ is so.
\end{proof}

\begin{remark} In part $(ii)$ of the above inclusion result, we may replace the condition that $A\subseteq B$ is $\mathfrak A$-extendable with any of the conditions that $\mathfrak A$ is injective and $A$ is $\mathfrak A$-nuclear or that $A$ is a hereditary $C^*$-subalgebra of $B$ (we still need the other inclusion to be $\mathfrak A$-extendable). To see this in the first case, identify $A$ with $\pi(A)$ and use $\mathfrak A$-nuclearity to approximately decompose $\pi_A={\rm id}_A$ via  c.c.p. admissible maps $\varphi_n: A\to \MM$ and $\psi_n:\MM\to \pi_A(A)$, in the point norm topology. Use injectivity of $\mathfrak A$ to extend $\varphi_n$ to a c.c.p. module map $\tilde\varphi_n$ on $B$ and choose $\varphi: B\to \pi_A(A)^{''}\subseteq \mathbb B(H_\mathfrak A)$ to be a point-ultraweak cluster point of the sequence $\pi_A\circ\psi_n\circ\tilde\varphi_n$, and continue the proof of part $(ii)$ as above. In the second case, let $\{e_i\}$ be a bounded approximate identity of $A$ and consider the admissible map ad$_{e_i}:B\to A; \ b\mapsto e_ibe_i$. As in the previous case, for a point-ultraweak cluster point $\varphi: B\to \pi_A(A)^{''}\subseteq \mathbb B(H_\mathfrak A)$ of the net  $\pi_A\circ$ad$_{e_i}$, one could proceed as before.
\end{remark}

Next, we want to prove a module version of so called ''the trick", used frequently in the literature of finite dimensional approximation \cite[3.6.5]{bo}. But first we need the following module version of the ``Arveson extension theorem''.  

\begin{lemma}[Extension of module maps] \label{ext}
	Let $\mathfrak A$ be a unital injective $C^*$-algebra and $A$ be a $C^*$-algebra and right $\mathfrak A$-module with compatible actions, and $E\subseteq A$ be an operator subsystem  and submodule. Assume that the algebra $\mathbb B(H\otimes \mathfrak A)$ of adjointable operators is a von Neumann algebra for a Hilbet space $H$. Then each c.c.p. module map $\phi: E\to \mathbb B(H\otimes \mathfrak A)$ extends to  a c.c.p. module map  $\tilde\phi: A\to \mathbb B(H\otimes \mathfrak A)$. When $\phi$ is u.c.p., the extension $\tilde\phi$ could also be chosen to be u.c.p.
\end{lemma}
	\begin{proof} By Lemma \ref{fr}, there is a net $(q_i)$ of finite rank projections, which SOT-increase to $1$ and $q_i\mathbb B(H_\mathfrak A)q_i\simeq \mathbb M_{n_i}(\mathfrak A)$. Then it follows from the proof of Lemma \ref{fr}, that the maps $\phi_i: E\to \mathbb M_{n_i}(\mathfrak A); \ x\mapsto q_i\phi(x)q_i$ are c.c.p. module maps. Since $\mathbb M_{n_i}(\mathfrak A)$ is an injective $\mathfrak A$-module, each $\phi_i$ has a c.c.p. module map extension $\tilde\phi_i: A\to \mathbb M_{n_i}(\mathfrak A)$, which could be regarded as a net of maps into the von Neumann algebra $B(H_\mathfrak A)$. Use \cite[1.3.7]{bo} to get a cluster point $\tilde \phi$ in the point-utraweak-topology, which is then the required extension of $\phi$. The last assertion is immediate from the proof. 
	\end{proof}

\begin{lemma}[The Trick] \label{trick}
Let $\mathfrak A$ be an injective $C^*$-algebra such that $\mathbb B(H_\mathfrak A)$ is a von Neumann algebra, for a Hilbet space $H$. Let $B$ and $C$ be $C^*$-algebras and right $\mathfrak A$-modules with compatible actions, and $A\subseteq B$ be a $C^*$-subalgebra and submodule. Let $\tilde\varrho$ be a module norm on $B\odot_{\mathfrak A} C^{op}$ and $\varrho$ be its restriction to $A\odot_{\mathfrak A} C^{op}$. Let $\pi_A: A\to \mathbb B(H_\mathfrak A)$ and  $\pi_C: C\to \mathbb B(H_\mathfrak A)$ be commuting representations and module maps in a left Hilbert $\mathfrak A$-module $H_\mathfrak A=H\otimes \mathfrak A$, satisfying
$$ \pi_A(a\cdot\alpha) \pi_C(c)=\pi_A(a) \pi_C(c\cdot\alpha),$$
for each $\alpha\in \mathfrak A, a\in A, c\in C,$  such that (the lift of) the corresponding product representation $\pi_A\times\pi_C: A\odot_{\mathfrak A} C^{op}\to \mathbb B(H_\mathfrak A)$ is $\varrho$-continuous, then $\pi_A$ has a c.c.p. module map extension $\varphi: B\to \pi_C(C)^{'}$.
\end{lemma}
\begin{proof}
First we assume that $\mathfrak A$ is unital and $A, B$ and $C$ are unital algebras and neo-unital modules with $1_A=1_B$. For the  extension $\pi_A\times_{\varrho}\pi_C: A\otimes_{\mathfrak A}^{\varrho} C^{op}\to \mathbb B(H_\mathfrak A)$, since $A\otimes_{\mathfrak A}^{\varrho} C^{op}\subseteq B\otimes_{\mathfrak A}^{\tilde\varrho} C^{op}$, by the unital case of Lemma \ref{ext}, there is a u.c.p. module map $\Phi: B\otimes_{\mathfrak A}^{\tilde\varrho} C^{op}\to \mathbb B(H_\mathfrak A)$, extending $\pi_A\times_{\varrho}\pi_C$. Put $\varphi=\Phi(\cdot\otimes 1_C)$. Then $\Phi(1_B\otimes c)=\pi_C(c)$, for each $c\in C$, hence $\mathbb C1_B\otimes C^{op}$ lies in the multiplicative domain of $\Phi$. Therefore,
\begin{align*}
\varphi(b)\pi_C(c)&=\Phi(b\otimes 1_C)\Phi(1_B\otimes c)\\
 &=\Phi\big((b\otimes 1_C)(1_B\otimes c)\big)\\
 &=\Phi\big((1_B\otimes c)(b\otimes 1_C)\big)\\
 &=\Phi(1_B\otimes c)\Phi(b\otimes 1_C)\\
&= \pi_C(c)\varphi(b),
\end{align*}
for each $b\in B, c\in C$, that is, $\varphi(B)\subseteq\pi_C(C)^{'}$.

When $\mathfrak A$ is unital, but $A, B$ and $C$ are not unital, by the discussion before Proposition \ref{iso} on extending norms on unitizations, one could deduce the result from the unital case  (since representations on ideals extend to representations of the algebra and the extension remains a module map if the original map is so).  Finally, one could extend the module actions of $\mathfrak A$ to that of $\mathfrak A\oplus \mathbb C$ to handle the case where $\mathfrak A$ is not unital.
\end{proof}

We say that an inclusion $A\subseteq B$ is $\mathfrak A$-{\it weakly injective} if there is a c.c.p. admissible map $\varphi: B\to A^{**}$ extending the identity on $A$. As the first application of the trick, we show the following result.

\begin{proposition}\label{winj}
For an inclusion $A\subseteq B$, the following are equivalent:

$(i)$ the inclusion is $\mathfrak A$-extendable,

$(ii)$ the inclusion is $\mathfrak A$-weakly injective,

$(iii)$ for every  $C^*$-algebra $C$ and right $\mathfrak A$-module with compatible action, we have the canonical inclusion  $A\otimes_{\mathfrak A}^{max} C^{op}\subseteq B\otimes_{\mathfrak A}^{max} C^{op}$.
\end{proposition}
\begin{proof}
Given a non degenerate faithful module representation $\pi: A\to \mathbb B(H_\mathfrak A),$ we may identify $\pi(A)$ with $A$. Choose a central projection $z\in A^{**}$ such that $A^{''}=zA^{**}z$ and $zaz=a$, for $a\in A$ \cite[3.7.7]{ped}. Now if the inclusion is $\mathfrak A$-weakly injective and $\varphi: B\to A^{**}$ is a c.c.p. admissible map  extending the identity on $A$, ad$_z\circ\varphi: B\to A^{''}$ is a c.c.p.  admissible. extension of id$_A$, that is,     $(ii)$ implies $(i)$. The converse implication is even easier as a copy of $A^{''}$ is in $A^{**}$. Also $(i)$ implies $(iii)$ by Proposition \ref{incl}. Finally, $(iii)$ implies $(i)$ by the trick applied to $C=\pi(A)^{'}$ and commuting representations $\pi: A\to \mathbb B(H_\mathfrak A)$ and $\iota: C\hookrightarrow \mathbb B(H_\mathfrak A)$.
\end{proof}

We say that $A$ has $\mathfrak A$-{\it weak expectation property} (or simply $A$ has $\mathfrak A$-WEP) if for every  module inclusion $A^{**}\subseteq \mathbb B(H_\mathfrak A)$, for a Hilbert $\mathfrak A$-module of the form $H_\mathfrak A$ with $H$  a Hilbert space, there is a u.c.p. admissible map $\varphi: \mathbb B(H_\mathfrak A)\to A^{**}$ extending the identity on $A$. The classical WEP is due to Christopher Lance \cite{lan2}. The next lemma shows that $\mathfrak A$-WEP is independent of the choice of the faithful representation.

\begin{lemma}\label{wep}
Assume that $\mathfrak A$ is unital. The following are equivalent:

$(i)$ $A$ has $\mathfrak A$-WEP,

$(ii)$ for each inclusion $A\subseteq B$ (as a subalgebra and submodule) and each $C^*$-algebra and right $\mathfrak A$-module $C$  with compatible action, we have the canonical inclusion  $A\otimes_{\mathfrak A}^{max} C^{op}\subseteq B\otimes_{\mathfrak A}^{max} C^{op}$.
\end{lemma}
\begin{proof}
If $(i)$ holds with a u.c.p. admissible map $\varphi: \mathbb B(H_\mathfrak A)\to A^{**}$, extending the identity on $A$, for the inclusion $A\subseteq B$, using the Arveson extension type theorem (Lemma \ref{ext}) there is a c.c.p. module map $\psi: B\to \mathbb B(H_\mathfrak A)$. Now Proposition \ref{winj} applied to $\varphi\circ\psi$ gives $(ii)$. When $A$ is unital, the converse follows from the trick applied to $B=\mathbb B(H_\mathfrak A)$ and $C=(A^{**})^{'}\subseteq \mathbb B(H_\mathfrak A)$. The non unital case follows by replacing $A$ with $A\oplus \mathfrak A$ and noting that $(A\oplus \mathfrak A)^{**}=A^{**}\oplus \mathfrak A^{**}$.
\end{proof}

\begin{proposition}\label{nuclear}
Let $A$ and $B$ be $C^*$-algebras and right $\mathfrak A$-modules with compatible actions and $\theta: A\to B$ be an $\mathfrak A$-nuclear module map. Then for each $C^*$-algebra and right $\mathfrak A$-module $C$ with compatible action, the map $\theta\otimes_{max}$id$_C: A\otimes_{\mathfrak A}^{max} C^{op}\to  B\otimes_{\mathfrak A}^{max} C^{op}$ factors through $ A\otimes_{\mathfrak A}^{min} C^{op}$.
\end{proposition}
\begin{proof}
Take c.c.p. admissible maps $\varphi_n: A\to \MM$ and $\psi_n:\MM\to B$ such that $$\psi_n\circ\varphi_n\to \theta,$$ in the point norm topology. Since there is a unique $C^*$-norm on $\MM\odot_{\mathfrak A} C=\mathbb M_{k(n)}(C)$, it follows that $$\varphi_n\otimes_{max}{\rm id}_C-(\varphi_n\otimes_{max}{\rm id}_C)\circ q\to 0,$$ in the point norm topology, where $q: A\otimes_{\mathfrak A}^{max} C^{op}\to  A\otimes_{\mathfrak A}^{min} C^{op}$ is the canonical quotient module map. Consider
$$\Psi_n:=(\psi_n\otimes_{max}\text{id}_C)\circ(\varphi_n\otimes_{min}\text{id}_C): A\otimes_{\mathfrak A}^{min} C^{op}\to B\otimes_{\mathfrak A}^{max} C^{op},$$
then $\Psi_n\to \theta\otimes$id$_C$ on $A\odot_{\mathfrak A} C^{op}$ in the point norm topology, thus $$\theta\otimes{\rm id}_C: A\odot_{\mathfrak A} C^{op}\to  B\odot_{\mathfrak A} C^{op}$$ is contractive with respect to the $min$ and $max$ module norms on the domain and range, respectively, and so extends to a contractive module map $$\Psi: A\otimes_{\mathfrak A}^{min} C^{op}\to B\otimes_{\mathfrak A}^{max} C^{op},$$ and $\Psi_n-\Psi\to 0$ in the point norm topology, and so $\Psi$ is a c.c.p. admissible map, and moreover we have the factorization $\theta\otimes_{max}$id$_C=\Psi\circ q$.
\end{proof}

\begin{corollary}\label{unique}
If $A$ is an $\mathfrak A$-nuclear, for each $C^*$-algebra and right $\mathfrak A$-module $C$ with compatible action, the max and min module tensor products $A\otimes_{\mathfrak A}^{max} C^{op}$ and $A\otimes_{\mathfrak A}^{min} C^{op}$ are isometrically isomorphic as $C^*$-algebras and modules. In particular, there is a unique module norm on $A\odot_{\mathfrak A} C^{op}$.
\end{corollary}
 \begin{proof}
Apply the above proposition to $\theta=$id$_A$.
\end{proof}

The above result extends one side of the well-known equivalence of nuclearity of a $C^*$-algebra $A$ with nuclearity of id$_A$. We shall prove the other side later

\begin{corollary}
If $A\subseteq \mathbb B(H_\mathfrak A)$ is an $\mathfrak A$-exact subalgebra and right submodule, then for each $C^*$-algebra and right $\mathfrak A$-module $C$ with compatible action, the restriction of the max module norm on $\mathbb B(H)\odot_{\mathfrak A} C^{op}$ to $A\odot_{\mathfrak A} C^{op}$ is the min module norm.
\end{corollary}
\begin{proof}
The above proposition applied to the $\mathfrak A$-nuclear inclusion $A\subseteq \mathbb B(H_\mathfrak A)$  gives a c.c.p. module map $\Psi: A\otimes_{\mathfrak A}^{min} C^{op}\to \mathbb B(H_\mathfrak A)\otimes_{\mathfrak A}^{max} C^{op},$ extending the identity map on $A\odot_{\mathfrak A} C^{op}$. This is an injective $*$-homomorphism on a dense subset of the module $min$ tensor product, and so it remains injective on the $min$ completion (just by the definition of the $min$ module norm).
\end{proof}

We observed that the $min$ module norm behaves better in preserving inclusions. Now we show that the $max$ module norm is better in preserving exact sequences. As in the classical case \cite[3.7.1, 3.7.2]{bo}, the proof of the next result is based on the fact that, for $C^*$-algebras and right $\mathfrak A$-modules $A$ and $B$  with compatible actions and closed ideal and submodule $J\unlhd A$, there are module norms $\varrho$ and $\delta$ such that
\begin{equation}\label{eq}
\frac{A\otimes_{\mathfrak A}^{max} B^{op}}{J\otimes_{\mathfrak A}^{max} B^{op}}\cong (A/J)\otimes_{\mathfrak A}^{\varrho} B^{op},\ \ \frac{A\otimes_{\mathfrak A}^{min} B^{op}}{J\otimes_{\mathfrak A}^{min} B^{op}}\cong (A/J)\otimes_{\mathfrak A}^{\delta} B^{op},
\end{equation}

\noindent which is observed as follows: Consider the exact sequence
$$0\to J\to A\to A/J\to 0$$
with arrows both $*$-homomorphisms and module maps, then we get the exact sequence
$$0\to J\odot_{\mathfrak A}B^{op}\to A\odot_{\mathfrak A}B^{op}\to (A/J)\odot_{\mathfrak A}B^{op}\to 0$$
and an embedding with dense range
$$(A/J)\odot_{\mathfrak A}B^{op}\cong \frac{A\odot_{\mathfrak A} B^{op}}{J\odot_{\mathfrak A} B^{op}}\hookrightarrow\frac{A\otimes_{\mathfrak A}^{max/min} B^{op}}{J\otimes_{\mathfrak A}^{max/min} B^{op}}.$$
Now the required module norms $\varrho$ and $\delta$ on $(A/J)\odot_{\mathfrak A}B^{op}$ are just the restriction of the quotient norms.

We say that the exact sequence
$$0\to J\rightarrow A\xrightarrow{\pi} A/J\to 0$$
is {\it locally $\mathfrak A$-split} if for each finitely generated operator subspace and submodule $E\subseteq A/J$ there is a c.p. module map $\sigma: E\to A$ with $\pi\circ\sigma=$id$_E$. The next result is proved as in the classical case \cite[section 3.7]{bo}.

\begin{proposition}\label{exseq}
With the above notations, if
$$0\to J\to A\to A/J\to 0$$
is an exact sequence with arrows both $*$-homomorphisms and module maps, then for each $C^*$-algebra and right $\mathfrak A$-module $B$  with compatible action,

$(i)$ the sequence
$$0\to J\otimes_{\mathfrak A}^{max}B^{op}\to A\otimes_{\mathfrak A}^{max}B^{op}\to (A/J)\otimes_{\mathfrak A}^{max}B^{op}\to 0$$
is exact,

$(ii)$ the sequence
$$0\to J\otimes_{\mathfrak A}^{min}B^{op}\to A\otimes_{\mathfrak A}^{min}B^{op}\to (A/J)\otimes_{\mathfrak A}^{min}B^{op}\to 0$$
is exact iff the module norm $\delta$ defined in (\ref{eq}) is the $min$ module norm. In particular, this holds when $A/J$ or $B$ is $\mathfrak A$-nuclear or the original exact sequence is locally $\mathfrak A$-split.
\end{proposition}

As in the classical case \cite[3.7.8]{bo}, we could get $(ii)$ above even if $B$ is $\mathfrak A$-exact.

\begin{proposition}\label{exseq2}
The conclusion of part $(ii)$ above also holds if $\mathfrak A$ is nuclear and $B$ is $\mathfrak A$-exact.
\end{proposition}
\begin{proof}
We need to show that the kernel of  $A\otimes_{\mathfrak A}^{min}B^{op}\to (A/J)\otimes_{\mathfrak A}^{min}B^{op}$ is contained in $J\otimes_{\mathfrak A}^{min}B^{op}$. Take any faithful representation and module map $B\xrightarrow{\pi} \mathbb B(H_\mathfrak A)$ and approximately decompose it through $\MM$ with c.c.p. admissible maps  $\varphi_n: B\to \MM$ and $\psi_n:\MM\to \mathbb B(H_\mathfrak A)$ and consider a diagram of three rows, where the first row is the exact sequence in part $(ii)$ above and the second and third rows are the same sequence with $B$ replaced with $\MM$ and $\mathbb B(H_\mathfrak A)$, respectively, and the vertical maps between the first and second row are the maps of the form id$\otimes_{min}\varphi_n$ and those between the second and third row are of the form id$\otimes_{min}\psi_n$, with id being id$_J$, id$_A$ or id$_{(A/J)}$ in the first to the third column. Since $\mathfrak A$ is nuclear, the second row is exact, hence for each $x$ in the kernel of  $A\otimes_{\mathfrak A}^{min}B^{op}\to (A/J)\otimes_{\mathfrak A}^{min}B^{op}$, id$_J\otimes_{min}(\psi_n\circ\varphi_n^{op})(x)$  is in $J\otimes_{\mathfrak A}^{min}\mathbb B(H_\mathfrak A)^{op}$, and so id$_J\otimes_{min}\pi^{op}(x)$  is in $(J\otimes_{\mathfrak A}^{min}\mathbb B(H_\mathfrak A)^{op})\cap(A\otimes_{\mathfrak A}^{min}\pi(B)^{op})=J\otimes_{\mathfrak A}^{min}\pi(B)^{op}$. Therefore, $x\in J\otimes_{\mathfrak A}^{min}B^{op}$, as $\pi$ is faithful.
\end{proof}

\begin{lemma}\label{mincont}
Let $A$ be $C^*$-algebra and a module with compatible actions, $B$ be a von Neumann algebra and a module with normal compatible action, and $\theta:A\to B$ be a c.p. module map. Consider  a  faithful representation and module map $B\subseteq \mathbb B(X)$ in a  $\mathfrak A$-Hilbert module $X$. If
$$\theta\times \mathrm{id}_{B^{'op}}: A\odot_{\mathfrak A}B^{'op}\to \mathbb B(X)$$
is min-continuous and $\pi: B\to \mathbb B(H_\mathfrak A)$ is any normal representation and module map in an $\mathfrak A$-Hilbert module $H_\mathfrak A:=H\otimes \mathfrak A$, where $H$ is a Hilbert space, then
$$(\pi\circ\theta)\times \mathrm{id}_{\pi (B)^{'op}}: A\odot_{\mathfrak A}\pi(B)^{'op}\to \mathbb B(\bar H_\mathfrak A)$$
is also min-continuous.
\end{lemma}
\begin{proof}
The normal representation $\pi$ is implemented via cut-down by a projection in the commutant of the copy of $B$ in $\mathbb B(\bar X\otimes_{\mathfrak A}H_\mathfrak A)$. Therefore, we only need to check that
$$(\theta\otimes_{min} 1_{\mathbb B(H_\mathfrak A)}) \times \mathrm{id}_{B^{'}\bar\otimes_{\mathfrak A} \mathbb B(H_\mathfrak A)^{op}}: A\odot_{\mathfrak A}(B^{'}\bar\otimes_{\mathfrak A}\mathbb B(H_\mathfrak A)^{op})\to \mathbb B(\bar X\otimes_{\mathfrak A}H_\mathfrak A)$$
is $min$-continuous. This is done, like in the classical case \cite[3.8.4]{bo}, by cutting down $\mathbb B(H_\mathfrak A)$ by a net of finite rank projections and taking limit in SOT: First let us take an increasing  net of finite rank projections $(q_i)\in \mathbb B(H_\mathfrak A)$, commuting with the action as in Lemma \ref{fr},  and note that the map 
$$(\theta\otimes_{min} 1_{\mathbb B(q_iH_\mathfrak A)}) \times \mathrm{id}_{B^{'}\bar\otimes_{\mathfrak A} \mathbb B(q_iH_\mathfrak A)^{op}}: A\odot_{\mathfrak A}(B^{'}\bar\otimes_{\mathfrak A}\mathbb B(q_iH_\mathfrak A)^{op})\to \mathbb B(\bar X\otimes_{\mathfrak A}q_iH_\mathfrak A)$$
could be identified with the map
$$(\theta\times {\rm id}_{B^{'}}) \otimes \mathrm{id}_{\mathbb B(q_iH_\mathfrak A)^{op}}: (A\odot_{\mathfrak A}B^{'})\odot_{\mathfrak A}\mathbb B(q_iH_\mathfrak A)^{op}\to \mathbb B(\bar X\otimes_{\mathfrak A}q_iH_\mathfrak A),$$
where, 
$$(A\odot_{\mathfrak A}B^{'})\odot_{\mathfrak A}\mathbb B(q_iH_\mathfrak A)^{op}=(A\odot_{\mathfrak A}B^{'})\odot_{\mathfrak A}\mathbb M_{n_i}(\mathfrak A)^{op}=\mathbb M_{n_i}(A\odot_{\mathfrak A}B^{'})^{op},$$
which shows that min and max norms coincide on the first algebraic tensor product, and so $(\theta\times {\rm id}_{B^{'}}) \otimes \mathrm{id}_{\mathbb B(q_iH_\mathfrak A)^{op}}$ is min norm continuous with norm at most $\|\theta\times {\rm id}_{B^{'}}\|$. Finally, this net of maps, after identification, converges SOT to  $(\theta\otimes_{min} 1_{\mathbb B(q_iH_\mathfrak A)}) \times \mathrm{id}_{B^{'}\bar\otimes_{\mathfrak A} \mathbb B(q_iH_\mathfrak A)^{op}}$, and we are done.
\end{proof}

Now we are ready to prove a module version of a result of Kirchberg \cite[3.8.5]{bo}.

\begin{theorem}[Kirchberg]\label{kirchberg}
 Let $\mathfrak A$ be unital. Let $A$ be $C^*$-algebra and a module with compatible action, $B$ be a von Neumann algebra and a module with normal compatible central action, and $\theta:A\to B$ be a c.p. module map.  Then the following are equivalent:

 $(i)$\ $\theta$ is $\mathfrak A$-weakly nuclear,

 $(ii)$ the product map $\theta\times \mathrm{id}_{B^{'op}}: A\odot_{\mathfrak A}B^{'op}\to \mathbb B(H_\mathfrak A)$ is min-continuous, for each Hilbert space $H$ and faithful representation and module map  $B\subseteq \mathbb B(\bar H_\mathfrak A)$.
 \end{theorem}
  \begin{proof}
If $(i)$ holds and $\varphi_n:A\to\mathbb M_{k(n)}(\mathfrak A)$ and $\psi_n:\mathbb M_{k(n)}(\mathfrak A)\to B$ are c.c.p. admissible maps where $\psi_n\circ\varphi_n$ point-ultraweak approximates $\theta$, then since there is a unique $C^*$-norm on $\mathbb M_{k(n)}(\mathfrak A)\odot_{\mathfrak A} B^{'op}$, we may form the composition $\Phi_n$ of the maps $$\psi_n\times\mathrm{id}_{B^{'op}}: \mathbb M_{k(n)}(\mathfrak A)\otimes_{\mathfrak A}^{max}B^{'op}\to \mathbb B(H_{\mathfrak A})$$ and $$\varphi_n\otimes_{min}\mathrm{id}_{B^{'op}}: A\otimes_{\mathfrak A}^{min}B^{'op}\to\mathbb M_{k(n)}(\mathfrak A)\otimes_{\mathfrak A}^{min}B^{'op},$$ whose point-ultraweak cluster point $\Phi: A\otimes_{\mathfrak A}^{min}B^{'op}\to\mathbb B(H_{\mathfrak A})$ is a c.c.p. admissible map, extending $\theta\times\mathrm{id}_{B^{'op}}:A\odot_{\mathfrak A}B^{'op}\to\mathbb B(H_\mathfrak A)$.

Conversely, if $(ii)$ holds, by Lemma \ref{wnucl}, we need to show that $\theta$ is point-ultraweak close to $\mathfrak A$-factorable maps. Fix finite sets $F=\{a_1,\cdots,a_k\}\subseteq A$ and $S=\{\tau_1,\cdots,\tau_m\}\subseteq B_{*}$ and $\varepsilon>0$. Put $\tau=\frac{1}{m}(\tau_1+\cdots\tau_m)$. By Radon-Nikodym theorem \cite[3.8.3]{bo}, there are elements $c_1,\cdots, c_m\in \pi_\tau(B)^{'}$ with $\tau_j=\langle\pi_\tau(\cdot)c_j\hat 1_B, \hat 1_B\rangle$, for $1\leq j\leq m$, where $\pi_\tau: B\to\mathbb B(L^2(B,\tau))$ is the GNS construction of $\tau$. We may regard $\pi_\tau$ as a map into $\mathbb B(L^2(B,\tau)\otimes \mathfrak A)$. By Lemma \ref{mincont}, we get the product map $$(\pi_\tau\circ\theta)\times \mathrm{id}_{B^{'op}}: A\odot_{\mathfrak A}^{min}B^{'op}\to\mathbb B(L^2(B,\tau)\otimes \mathfrak A).$$
Let us define the state $\tilde\tau$ on $A\odot_{\mathfrak A}^{min}\pi_\tau(B)^{'op}$ by $\tilde\tau(a\otimes c)=\langle\pi_\tau(\theta(a))c\hat 1_B, \hat 1_B\rangle$, then $\tilde\tau(a\otimes c_j)=\tau_j(\theta(a))$, for $1\leq j\leq m$.

Let us take an essential faithful representation $A\subseteq \mathbb B(K)$ in an Hilbert space $K$ (inflate if necessary, to get an essential representation), and apply the Glimm's Lemma \cite[1.4.11]{bo} to the faithful representation $A\odot_{\mathfrak A}^{min}\pi_\tau(B)^{'op}\subseteq \mathbb B(K\otimes L^2(B,\tau))$ to find $\{\xi_1,\cdots,\xi_n\}\subseteq K$ and $\{b_1,\cdots, b_n\}\subseteq B$ with
$$|\tilde\tau(a_i\otimes c_j)-\langle a\otimes c_j(\xi),\xi\rangle|<\varepsilon\ \ (1\leq i\leq k, 1\leq j\leq m),$$
for $\xi=\sum_{i=1}^n\xi_i\otimes \hat b_i$. Let $p\in\mathbb B(K)$ be the orthogonal projection onto the subspace generated by $\xi_i$'s and consider the c.p. map $\psi: p\mathbb B(K)p\to B$ defined by $\psi(\xi_i\otimes\xi_\ell)=b_i^*b_\ell$, extended linearly and continuously, then
\begin{align*}
\tau_j(\psi(pap))&=\sum_{i,\ell=1}^{n}\langle a\xi_\ell, \xi_i\rangle\tau_j(b_i^*b_\ell)\\
&=\sum_{i,\ell=1}^{n}\langle a\xi_\ell, \xi_i\rangle\langle\pi_\tau(b_i^*b_\ell)c_j\hat 1_B, \hat 1_B\rangle\\
&=\sum_{i,\ell=1}^{n}\langle a\xi_\ell, \xi_i\rangle\langle c_j\hat b_\ell, b_i^*\rangle\\
&=\langle a\otimes c_j(\xi),\xi\rangle,
\end{align*}
for $1\leq j\leq m$, with $\xi=\sum_{i=1}^n\xi_i\otimes \hat b_i$ as above. Thus $|\tau_j(\theta(a))-\tau_j(\psi(pap))|<\varepsilon,$ for each $j$. Since $\mathfrak A$ acts centrally on $B$, the map $\rho: \mathfrak A\to B; \ \alpha\mapsto 1_B\cdot\alpha$ is into the center of $B$,  in particular, $\psi$ and $\rho$ have commuting ranges, and we may form $\tilde\psi:=\psi\times\rho: \mathbb B(pK)\otimes \mathfrak A\to B$. We want to compose this with the c.p. map $\tilde\phi: A\to  \mathbb B(pK)\otimes \mathfrak A; \ a\mapsto (p\otimes 1)(a\times 1_\mathfrak A)(p\otimes 1)$, to estimate $\theta$. This finishes the proof, as  $\mathbb B(pK)\otimes \mathfrak A
=\mathbb M_r(\mathfrak A)$, where $r={\rm rank}(p)$, by Lemma \ref{fr}. We have, $$\tilde\psi\circ\tilde\phi(a)=(\psi\times\rho)(pap\otimes 1_\mathfrak A)=\psi(pap)\rho(1_\mathfrak A)=\psi(pap)1_B=\psi(pap),$$
thus the above estimate could be rewritten as, $|\tau_j(\theta(a))-\tau_j(\tilde\psi\circ\tilde\phi(a))|<\varepsilon,$ for each $j$, as required.  
\end{proof}

\begin{corollary} Let $\mathfrak A$  be a von Neumann algebra and $A$ be a von Neumann algebra and $\mathfrak A$-module with compatible normal central action. Then the following are equivalent:

$(i)$ $A$ is $\mathfrak A$-semidiscrete,

$(ii)$ $\mathrm{id}_{A}\times \mathrm{id}_{A^{'op}}: A\odot_{\mathfrak A}A^{'op}\to \mathbb B(H_\mathfrak A)$ is min-continuous, for each Hilbert space $H$ and faithful representation and module map $A\subseteq \mathbb B(\bar H_\mathfrak A)$,

$(iii)$ $A^{'}$ is $\mathfrak A$-semidiscrete.
\end{corollary}

As another application, we could prove the converse of Proposition \ref{nuclear}.

\begin{corollary}\label{nuclear2}
Let $A$ and $B$ be $C^*$-algebras and right $\mathfrak A$-modules with compatible actions and $\theta: A\to B$ be a module map. Assume that there is a faithful representation and module map $B^{**}\subseteq \mathbb B(H\otimes\mathfrak A)$, for some Hilbert space $H$. Then $\theta$ is $\mathfrak A$-nuclear iff for each $C^*$-algebra and right $\mathfrak A$-module $C$ with compatible action, the map $$\theta\otimes_{max}{\rm id}_{C^{op}}: A\otimes_{\mathfrak A}^{max} C^{op}\to  B\otimes_{\mathfrak A}^{max} C^{op}$$ factors through $ A\otimes_{\mathfrak A}^{min} C^{op}$.
\end{corollary}
\begin{proof}
Consider the embeddings  $B\subseteq B^{**}\subseteq \mathbb B(H_\mathfrak A)$. The map $$\theta\otimes_{max}{\rm id}_{B^{**'op}}: A\otimes_{\mathfrak A}^{max} B^{**'op}\to B\otimes_{\mathfrak A}^{max}B^{**'op}$$ factors through the corresponding $min$-tensor product, and composing with the continuous inclusion $\iota: B\otimes_{\mathfrak A}^{max} B^{**'op}\hookrightarrow\mathbb B(H_\mathfrak A)$, we get that the map $$\theta\times{\rm id}_{B^{**'op}}: A\odot_{\mathfrak A}B^{**'op}\to \mathbb B(H_\mathfrak A)$$ is $min$-continuous, and so $\theta: A\to B^{**}$ is  $\mathfrak A$-nuclear. Since the map already ranges in $B$, it is also $\mathfrak A$-nuclear as a map into $B$.
\end{proof}

Now we are ready to prove the module version of a result, obtained in the classical case independently by Choi-Effros \cite{ce2} and Kirchberg \cite{k}.

\begin{theorem}[Choi-Effros, Kirchberg]\label{kirchberg2}
Let $\mathfrak A$ be an injective $C^*$-algebra and $A$ be $C^*$-algebra and a module with compatible action, with a faithful representation and module map $A^{**}\subseteq \mathbb B(H_\mathfrak A)$, for some Hilbert space $H$. Then, the following are equivalent:

 $(i)$\ $A$ is $\mathfrak A$-nuclear,

 $(ii)$ for any $C^*$-algebra and module $B$ with compatible action, there is a unique module norm on $A\odot_{\mathfrak A} B$.
 \end{theorem}
  \begin{proof} By Corollary \ref{unique}, $(i)$ implies $(ii)$. Conversely consider the  inclusion and module map $\iota: A\to A^{**}\subseteq \mathbb B(H_\mathfrak A)$, then the  map $\iota\times \mathrm{id}: A\odot_{\mathfrak A}A^{**'op}\to \mathbb B(H_\mathfrak A)$ is $max$-continuous, and so by $(ii)$ it is also $min$-continuous, hence $\iota$ is $\mathfrak A$-weakly nuclear, by the above theorem. Now $(i)$ follows from the remark after Proposition \ref{inj}.
\end{proof}

Now we turn to $\mathfrak A$-exact algebras, and extend a well known result of Kirchberg on exact $C^*$-algebras \cite{k2}. We need some preparation first.

\begin{lemma}\label{subsystem}
Let  $C\subseteq A$ be an operator subsystem  and submodule and $J\unlhd B$ be a closed ideal and submodule, then there is an isometric inclusion $$(C\otimes_{\mathfrak A}^{min} B^{op})/ (C\otimes_{\mathfrak A}^{min} J^{op})\subseteq (A\otimes_{\mathfrak A}^{min} B^{op})/(A\otimes_{\mathfrak A}^{min} J^{op}),$$ preserving the module actions.
\end{lemma}
\begin{proof} The inclusion is isometric as in the classical case \cite[3.9.2]{bo}. It is also a module map, as the inclusion $C\otimes_{\mathfrak A}^{min} B^{op}\subseteq A\otimes_{\mathfrak A}^{min} B^{op}$ is so.
\end{proof}

\begin{lemma}\label{subsystem2}
Let  $J\unlhd B$ be a closed ideal and submodule, then $$(A\otimes_{\mathfrak A}^{min} B^{op})/(A\otimes_{\mathfrak A}^{min} J^{op})\cong A\otimes_{\mathfrak A}^{min} (B/J)^{op},$$ canonically, iff the same holds for $A$ replaced by any operator subsystem  and finitely generated submodule $C\subseteq A$.
\end{lemma}
\begin{proof} If there is a canonical isomorphism as above for all operator subsystems  and finitely generated submodules $C\subseteq A$, then the same holds for $A$, since the union of all quotients $(C\otimes_{\mathfrak A}^{min} B^{op})/(C\otimes_{\mathfrak A}^{min} J^{op})$ with $C$ as above is dense in $(A\otimes_{\mathfrak A}^{min} B^{op})/(A\otimes_{\mathfrak A}^{min} J^{op}).$ The converse follows from the above lemma.
\end{proof}

\begin{lemma}\label{subsystem}
Let $\mathfrak A$ be a unital $C^*$-algebra,   $C\subseteq A$ be an operator subsystem and a finitely generated  submodule, and $B_n$ be unital $C^*$-algebras and unital modules with compatible actions and put $B_0:=\bigoplus_{n} B_n\subseteq B:=\prod_{n} B_n$, then there is a u.c.p. module map and isometric isomorphism of $C^*$-algebras $$C\otimes_{\mathfrak A}^{min} B^{op}\to  \prod_{n} (C\otimes_{\mathfrak A}^{min} B_n^{op})$$ sending $C\otimes_{\mathfrak A}^{min} B_0^{op}$ to $\bigoplus_{n} (C\otimes_{\mathfrak A}^{min} B_n^{op})$. In particular, there is a contractive module map $\prod_{n}(C\otimes_{\mathfrak A}^{min} B_n^{op})/\bigoplus_{n}(C\otimes_{\mathfrak A}^{min} B_n^{op})\to C\otimes_{\mathfrak A}^{min} (B/B_0)^{op}$, which is an isometric isomorphism of $C^*$-algebras when module $min$-tensoring with $C$ preserves the exactness of short exact sequences $$0\to J^{op}\to A^{op}\to (A/J)^{op}\to 0$$
with arrows both $*$-homomorphisms and module maps, for each $C^*$-algebra and right $\mathfrak A$-module $B$  with compatible action and each closed ideal and submodule $J$.
\end{lemma}
\begin{proof}
The proof goes  as in the classical case \cite[33.9.4-5]{bo} with  representations and module maps $B_n^{op}\subseteq \mathbb B(\bar H_n\otimes\mathfrak A)$ and $A \subseteq \mathbb B(K\otimes\mathfrak A)$, for Hilbert spaces $H_n$ and  $K$ via the identification $\mathbb B(K\otimes (\oplus_n \bar H_n)\otimes\mathfrak A)\cong \mathbb B(\oplus_n(K\otimes \bar H_n\otimes\mathfrak A).$
\end{proof}

Next, let $C$ be a separable $C^*$-algebra and right $\mathfrak A$-module with compatible action and $C\subseteq \mathbb B(H_\mathfrak A)$ be a faithful representation and module map, for a separable  Hilbert space  $H$, and $(q_i)$ be an increasing sequence of finite rank projections in $\mathbb B(H_\mathfrak A)$ converging strongly to the identity  and identify $q_i\mathbb B(H_\mathfrak A)q_i$ with $\mathbb M_{n_i}(\mathfrak A)$, as in Lemma \ref{fr}. We use the notations, $C_i:=q_iCq_i$ and $\mathrm{ad}_i: x\mapsto q_ixq_i$.

\begin{lemma}\label{ad}
If $\mathfrak A$ is an exact unital $C^*$-algebra,  $C$ is a unital module such that  $min$-tensoring with $C$ preserves the exactness of short exact sequences $$0\to J^{op}\to A^{op}\to (A/J)^{op}\to 0,$$ as above, then  $\big\|\mathrm{ad}_i^{-1}|_{C_i}\big\|_{cb}\to 1$, as $i\to\infty$ and there are u.c.p. module maps $\psi_i: \mathbb B(q_iH_\mathfrak A)\to A$ with $\big\|(\psi_i-\mathrm{ad}_i^{-1})|_{C_i}\big\|_{cb}\to 0,$ as $i\to\infty$.
\end{lemma}
\begin{proof}
If the assertion does not hold, we may assume that the above $cb$-norm tends to some $\lambda>1$. By the definition of the $cb$-norm, there is an increasing  sequence $\{k_i\}$ of positive integers and norm one elements $c_i\in\mathbb M_{k_i}(C)$ with $\|(\mathrm{ad}_i\otimes\mathrm{id}_{k_i})(c_i)\|\to\frac{1}{\lambda}$. Let $c=(c_n)+\bigoplus_i \mathbb M_{k_i}(C)$ be the corresponding coset and $\tilde c$ be the image of $c$ in $$C\otimes_{\mathfrak A}^{min} \big(\prod_i \mathbb M_{k_i}(\mathfrak A)/\bigoplus_i \mathbb M_{k_i}(\mathfrak A)\big)^{op}\subseteq \mathbb B(H\otimes \bar K\otimes \mathfrak A),$$ as in Lemma \ref{subsystem}, for some  Hilbert space $K$, and $\|\tilde c\|=\sup_k \|(\mathrm{ad}_k\otimes_{min} \mathrm{id})(\tilde c)\|$, where the right norm is in $M:=\mathbb M_{k}(\mathfrak A)\otimes_{\mathfrak A}^{min} (\prod_i \mathbb M_{k_i}(\mathfrak A)/\bigoplus_i \mathbb M_{k_i}(\mathfrak A))^{op}$. By exactness of $\mathfrak A$, it follows from  \cite[3.9.5]{bo} that
$$M^{op}\cong\big(\prod_i \mathbb M_{k}(\mathbb M_{k_i}(\mathfrak A))\big)/ \big(\bigoplus_i \mathbb M_{k}(\mathbb M_{k_i}(\mathfrak A))\big),$$
for each $k$, thus
\begin{align*}\|\tilde c\|&=\sup_k \|(\mathrm{ad}_k\otimes_{min} \mathrm{id})(\tilde c)\|\\&\leq\sup_{k}\limsup_{i} \|(\mathrm{ad}_k\otimes_{min} \mathrm{id}_{k_i})(c_i)\|\\&\leq\limsup_{i} \|(\mathrm{ad}_i\otimes_{min} \mathrm{id}_{k_i})(c_i)\|\\&=\frac{1}{\lambda}<\|c\|,\end{align*}
 a contradiction. 
 
 Next, let us prove the last statement. The proof given here adapts the argument of a similar statement in  \cite[B.11]{bo}. Recall that $C\subseteq \mathbb B(H_\mathfrak A)$ and $C_i:=q_iCq_i$, where $q_i$ is a finite rank projection with $q_i\mathbb B(H_\mathfrak A)q_i=\mathbb M_{n_i}(\mathfrak A)$. Consider the c.c. map $\theta_i:=$ad$_i^{-1}: \mathbb M_{n_i}(\mathfrak A)\to A$, then using an argument as in \cite[Lemma 3]{am}, we may identify the module map $\theta_i$ with a c.c. linear map, still denoted by $\theta_i: \mathbb M_{n_i}(\mathbb C)\to A$, defined by $\theta_i(e_{k\ell})=\theta_i(e_{k\ell}\otimes 1_\mathfrak A)$. Now it follows from \cite[3.9.7]{bo} that there is a net of u.c.p. maps $\psi_i: \mathbb M_{n_i}(\mathbb C)\to A$ with $\|\psi_i|_{C_i}-\theta_i\|_{cb}\to 0$, as $i\to \infty$. Going backwards by \cite[Lemma 3]{am}, we may regard $\psi_i$ as a  u.c.p. module map $\psi_i: \mathbb M_{n_i}(\mathfrak A)\to A$, and we have, 
 $$\big\|(\psi_i-\mathrm{ad}_i^{-1})|_{C_i}\big\|_{cb}\to 0,$$
as required.
\end{proof}

\begin{theorem}[Kirchberg]\label{kirchberg3}
Consider the following statements:

$(i)$ $A$ is $\mathfrak A$-exact,

$(ii)$ for any exact sequence
$$0\to J\to B\to B/J\to 0$$
with arrows both $*$-homomorphisms and module maps, and each object a $C^*$-algebra and right $\mathfrak A$-module with compatible action,
the sequence
$$0\to J\otimes_{\mathfrak A}^{min}A^{op}\to B\otimes_{\mathfrak A}^{min}A^{op}\to (B/J)\otimes_{\mathfrak A}^{min}A^{op}\to 0$$
is exact.

Then $(i)\Rightarrow (ii)$, when $\mathfrak A$ is nuclear, and $(ii)\Rightarrow (i)$, when $\mathfrak A$ is exact.
\end{theorem}
\begin{proof}
The first implication is proved in Proposition \ref{exseq2}. For the other implication, first assume that $\mathfrak A$ is unital and  $A$ is a separable $C^*$-algebra and a countably generated unital right $\mathfrak A$-module. Take an operator subsystem and finitely generated submodule $C\subseteq A$ and take $\psi_i: C_i\to A$ as in the above lemma. Let $A\subseteq \mathbb B(H_\mathfrak A)$ be a faithful representation and module map for a Hilbert space $H$, and use Lemma \ref{ad} to find u.c.p. module maps $\psi_i: \mathbb M_{n_i}(\mathfrak A)\to \mathbb B(H_\mathfrak A)$ and $\mathrm{ad}_i: C\to \mathbb M_{n_i}(\mathfrak A)$ with
$$\|\psi_i(\mathrm{ad}_i(x))-x\|=\|\psi_i(\mathrm{ad}_i(x))-\mathrm{ad}_i^{-1}(\mathrm{ad}_i(x))\|\to 0,$$
as $i\to\infty$. Since this could be done for each finitely generated submodule $C$, we conclude that $A$ is $\mathfrak A$-exact.
The non separable case follows from the fact that if all separable, countably generated unital submodules of $A$ are  $\mathfrak A$-exact, then so is $A$. Finally the non unital case follow by considering first $A\oplus\mathfrak A$ instead of $A$ and then $\mathfrak A\oplus \mathbb C$ instead of $\mathfrak A$.
\end{proof}

We conclude this paper by a concrete application of the results of this section to the $C^*$-algebra of an inverse semigroup. We refer the reader to \cite{dp} for more details. We know that for a discrete group $G$, the reduced $C^*$-algebra $C_r^*(G)$ is nuclear iff $G$ is amenable. We show that for an inverse semigroup $S$ with the subsemigroup $E_S$ of idempotents, the reduced $C^*$-algebra $C_r^*(S)$ is $C_r^*(E_S)$-nuclear iff $S$ is left amenable (see \cite{d} for the notion of amenability of semigroups.)

Consider the equivalence relation $s\sim t$ iff there is $e\in E_S$ with $es=et$ and let $[s]$ be the equivalence class of $s\in S$, then $G_S:=\{[s]: s\in S\}$ is the maximal group homomorphic image of $S$ and it is a classical result of Duncan and Namioka that $S$ is left amenable iff $G_S$ is amenable \cite{P}. The surjective semigroup homomorphism $s\in S\mapsto [s]\in G_S$ lifts to a surjective $C^*$-homomorphism $\pi_r: C_r^*(S)\to C_r^*(G_S)$ whose kernel is the closed ideal $J_r$ of $C_r^*(S)$ generated by elements of the form $\delta_{t}-\delta_{s}$, for $s,t\in S$ satisfying $es=et$ for some $e\in E_S$. Similarly, there is a surjective $C^*$-homomorphism $\pi_r: C^*(S)\to C^*(G_S)$ whose kernel is the closed ideal $J$ of $C^*(S)$ generated by elements of the above form. Also the left action $E_S\times S\to S; \ (e,s)\mapsto es$ lifts to a right action $\ell^1(S)\times E_S\to \ell^1(S)$ defined by $f\cdot e(s)=f(es)$ for $f\in\ell^1(S)$. This is continuous in the reduced $C^*$-norm and extends to a right action of $E_S$ on $C_r^*(S)$. Finally,  $C_r^*(S)$ becomes a right $C_r^*(E_S)$-module. If $B$ is any $C^*$-algebra and right $C_r^*(E_S)$-module with compatible actions then
$$C_r^*(S)\odot_{C_r^*(E_S)} B^{op}\cong (C_r^*(S)\odot B^{op})/I,$$
where $I$ is the ideal generated by elements of the form $\delta_{es}\otimes b-\delta_s\otimes b\cdot\delta_e$, for $e\in E_S, s\in S$ and $b\in B$. Let $I_{min}$ be the $min$-closure of $I$, then the above algebraic isomorphism lifts to a $C^*$-isomorphism and module map $C_r^*(S)\otimes_{C_r^*(E_S)}^{min} B^{op}\cong (C_r^*(S)\otimes_{min} B^{op})/I_{min}.$
A similar argument also works for the full $C^*$ algebras and $C^*(S)$ is a right $C^*(E_S)$-module, and $C^*(S)\otimes_{C_r^*(E_S)}^{max} B^{op}\cong (C^*(S)\otimes_{max} B^{op})/I_{max}.$
Note that since $E_S$ is an abelian semigroup, $C_r^*(E_S)$ is a commutative $C^*$-algebra, hence it is nuclear (and so exact). The spectrum of this commutative $C^*$-algebra is the space $\hat E_S$ of semi-characters on $E_S$.

In the next result, we assume that $S$ is countable to guarantee the existence of faithful representations needed for the minimal tensor norm. 

\begin{proposition} For a countable  inverse semigroup $S$ with the set of idempotents $E_S$ and maximal group homomorphic image $G_S$,

$(i)$  $C_r^*(S)$ is $C_r^*(E_S)$-nuclear iff $S$ is left amenable,

$(ii)$  if $C_r^*(S)$ is $C_r^*(E_S)$-exact then $G_S$ is exact.
\end{proposition}
\begin{proof}
$(i)$ Let $B$ be a separable $C^*$-algebra. To avoid complications with opposite algebras, we consider $B$ as a $C_r^*(E_S)$-module with trivial left action. With the above notations, let $es=et$, then since $\delta_e\cdot b=b$, we have
 $$(\delta_{t}-\delta_{s})\otimes b=(\delta_{t}\otimes \delta_e\cdot b-\delta_{et}\otimes b)-(\delta_{s}\otimes \delta_e\cdot b-\delta_{es}\otimes b)\in I_{min}.$$
Conversely, $\delta_{t}\otimes \delta_e\cdot b-\delta_{et}\otimes b=(\delta_{t}-\delta_{et})\otimes b\in J_r\otimes_{min} B$. Therefore,  $I_{min}=J_r\otimes_{min} B$ and \begin{align*}C_r^*(S)\otimes_{C_r^*(E_S)}^{min} B&\cong (C_r^*(S)\otimes_{min} B)/I_{min}\\&\cong(C_r^*(S)\otimes_{min} B)/(J_r\otimes_{min} B)\\&\cong(C_r^*(S)/J_r)\otimes_{min} B)\\&\cong C_r^*(G_S)\otimes_{min} B.\end{align*}
Similarly, $I_{max}=J\otimes_{max} B$ and $C^*(S)\otimes_{C^*(E_S)}^{max} B\cong  C^*(G_S)\otimes_{max} B.$ If $C_r^*(S)$ is $C_r^*(E_S)$-nuclear, then by Theorem \ref{kirchberg2}, $C^*(S)\otimes_{C^*(E_S)}^{max} B\cong C^*(S)\otimes_{C^*(E_S)}^{min} B$ and so $C_r^*(G_S)\otimes_{max} B\cong C_r^*(G_S)\otimes_{min} B$, hence by the classical Choi-Effros Kirchberg theorem \cite[3.8.7]{bo}, $C_r^*(G_S)$ is nuclear, thus $G_S$ is amenable and by a classical result of Duncan-Namioka \cite{P}, $S$ is left amenable.

Conversely, if $S$ is left amenable, then $G_S$ is amenable and $C_r^*(G_S)$ is nuclear. Take $B=C_r^*(E_S)$ in the above calculation to get  $$C_r^*(S)\otimes_{C_r^*(E_S)}^{min} C_r^*(E_S)\cong  C_r^*(G_S)\otimes_{min} C_r^*(E_S).$$
Note that here the isomorphism clearly preserves the module actions on both sides and so this is also an isomorphism of $C_r^*(E_S)$-modules, however, since $C_r^*(E_S)$ acts trivially on itself, and so the left hand side is not isomorphic to $C_r^*(S)$.
Now the identity map $\mathrm{id}: C_r^*(G_S)\to C_r^*(G_S)$ is point norm limit of $\psi_n\circ\varphi_n$ for some c.c.p. maps $\varphi_n: C_r^*(G_S)\to\mathbb M_{k_n}(\mathbb C)$ and $\psi_n: \mathbb M_{k_n}(\mathbb C)\to C_r^*(G_S)$. Tensoring with the identity map on $C_r^*(E_S)$, $\mathrm{id}_{C_r^*(G_S)}\otimes_{min}\mathrm{id}_{C_r^*(E_S)}$ is point-norm approximated by the composition $(\psi_n\otimes_{min}\mathrm{id}_{C_r^*(E_S)})\circ(\varphi_n\otimes_{min}\mathrm{id}_{C_r^*(E_S)})$, which gives an approximate factorization through $\mathbb M_{k_n}(\mathbb C)\otimes_{min}C_r^*(E_S)=\mathbb M_{k_n}(C_r^*(E_S))$, thus $C_r^*(G_S)\otimes_{min} C_r^*(E_S)$ is $C_r^*(E_S)$-nuclear, and so is $C_r^*(S)\otimes_{C_r^*(E_S)}^{min} C_r^*(E_S)$. Therefore, by Lemma \ref{tensor}, $C_r^*(S)$ is also $C_r^*(E_S)$-nuclear.

$(ii)$ Let $B$ be a separable $C^*$-algebra and $J$ be a closed ideal in $B$. Consider the exact sequence
$$0\to J\to B\to B/J\to 0$$
with arrows  $*$-homomorphisms. Regard this as an exact sequence of $C_r^*(E_S)$-modules with trivial actions. If $C_r^*(S)$ is $C_r^*(E_S)$-exact, then
the sequence
$$0\to C_r^*(S)\otimes_{C_r^*(E_S)}^{min}J\to C_r^*(S)\otimes_{C_r^*(E_S)}^{min}B\to C_r^*(S)\otimes_{C_r^*(E_S)}^{min}(B/J)\to 0,$$
which is the same as
$$0\to C_r^*(G_S)\otimes_{min}J\to C_r^*(G_S)\otimes_{min}B\to C_r^*(G_S)\otimes_{min}(B/J)\to 0,$$
is exact. By a classical result of Kirchberg \cite[3.9.1]{bo}, $C_r^*(G_S)$ is exact and so $G_S$ is an exact group.
\end{proof}

We don't know if the converse of $(ii)$ is true. If $G_S$ is exact, there is  a  Hilbert space $H$ and a faithful nuclear representation  $C_r^*(G_S)\hookrightarrow \mathbb B(H)$, and  we have the nuclear inclusion
 $$C_r^*(G_S)\otimes_{min} C_r^*(E_S)\hookrightarrow \mathbb B(H)\otimes_{min}C_r^*(E_S)$$
and the inclusion
$$\mathbb B(H)\otimes_{min}C_r^*(E_S) \hookrightarrow \mathbb B(H\otimes C_r^*(E_S)).$$
By an argument similar to that of part $(i)$ above,  the inclusion
$$C_r^*(S)\otimes_{C_r^*(E_S)}^{min}C_r^*(E_S)\hookrightarrow B(H\otimes C_r^*(E_S))$$
is $C_r^*(E_S)$-nuclear. If there is a c.p. module map inclusion
$$C_r^*(S)\hookrightarrow C_r^*(S)\otimes_{C_r^*(E_S)}^{min}C_r^*(E_S)$$
then $C_r^*(S)$ would be $C_r^*(E_S)$-exact.

\begin{corollary}
If $S$ is left amenable, then $C^*_r(S)$ has $C_r^*(E_S)$-WEP.
\end{corollary}
\begin{proof}
By the above proposition, $C^*_r(S)$ is $C_r^*(E_S)$-nuclear. Now the result follows from Lemma \ref{wep} (which also works for non unital $\mathfrak A$) and the Remark after Proposition \ref{incl}.
\end{proof}

Again we don't know if the converse holds. If $C^*_r(S)$ has $C_r^*(E_S)$-WEP, by Lemma \ref{wep}, the inclusion $\iota: C^*_r(S)\hookrightarrow \mathbb B(\ell^2(S))$ induces a module map inclusion $\iota\otimes\mathrm{id}: C^*_r(S)\otimes_{C_r^*(E_S)}^{max}C^*_r(S)\hookrightarrow \mathbb B(\ell^2(S))\otimes_{C_r^*(E_S)}^{max}C^*_r(S)$ (note that the left regular representation: $C_r^*(E_S)\to \mathbb B(\ell^2(E_S))\hookrightarrow \mathbb B(\ell^2(S))$ gives $\ell^2(S)$ the structure of an $\mathfrak A$-Hilbert space). If $\lambda$ and $\rho$ are the left and right regular representations of $S$, then $\lambda\times\rho: C^*_r(S)\otimes_{max}C^*_r(S)\to \mathbb B(\ell^2(S))$ does not satisfy the conditions of the trick in general. If $S$ is a Clifford semigroup, then $E_S$ lies is in the center of $S$ and for each $\xi\in\ell^2(S)$ we have
$$(\lambda(es)\rho(s)-\lambda(s)\rho(es))(\xi)(t)=\xi(s^*ets)-\xi(s^*tes)=0,$$
when $tt^*\leq ess^*$, and the left hand side is equal to $0-0=0$, otherwise. Applying the trick to $\pi_A=\lambda$, $\pi_C=\rho$,  $B=\mathbb B(\ell^2(S))$ and $\varrho=max$, we get a c.c.p. module map extension $\varphi: \mathbb B(\ell^2(S))\to \rho(C^*(S))^{'}$ of $\lambda$. Since the ranges of $\lambda$ and $\rho$ commute, we have $\rho(C^*(S))^{'}=\lambda(C^*(S))^{''}:=L(S)$, and $\varphi$ extends the identity map on $C^*_r(S)$. Hence $L(S)$ is an injective von Neumann algebra. We know that for an inverse semigroup $S$, this holds when all maximal subgroups of $S$ are amenable \cite[Theorem 4.5.2]{P}, but the converse is not known to be true in general \cite[Page 209]{P}. Also in general, the amenability of all maximal subgroups do not imply left amenability of $S$ (for instance, the free inverse semigroup on two generators has trivial maximal subgroups.) Here we have the advantage that $S$ is a Clifford semigroup and there is a norm one projection: $L(S)\to L(S_e)$ for any maximal subgroup $S_e=\{s\in S: ss^*=s^*s=e\}$, for $e\in E_S$. Hence $L(S_e)$ is injective and so $S_e$ is an amenable (discrete) group. This is known to be equivalent to the weak containment property for $S$, that is, to $C^*(S)=C^{*}_r(S)$ \cite{P2}.

\end{document}